\documentclass[10pt]{amsart} 
\usepackage{amssymb}
\usepackage{amsmath}
\usepackage{enumitem}
\usepackage{mathrsfs}
\usepackage[english]{babel}
\usepackage[all]{xy}
\usepackage{hyperref}
\usepackage{marginnote}
\usepackage{mathtools}

\usepackage{tikz-cd}
\usepackage{tikz}
\usepackage{fullpage}

\usepackage{stmaryrd}

\newtheorem{theorem}{Theorem}[section]
\newtheorem*{theorem*}{Theorem}
\newtheorem{lemma}[theorem]{Lemma}
\newtheorem{proposition}[theorem]{Proposition}
\newtheorem{corollary}[theorem]{Corollary}

\newtheorem{alphtheorem}{Theorem}
\newtheorem{alphcorollary}[alphtheorem]{Corollary}

\theoremstyle{definition}
\newtheorem*{ack}{Acknowledgements}

\newtheorem{remark}[theorem]{Remark}

\newtheorem{definition}[theorem]{Definition}
\newtheorem{setup}[theorem]{Setup}

\numberwithin{equation}{section} \numberwithin{figure}{section}

\DeclareMathOperator{\Aut}{Aut}

\DeclareMathOperator{\Spec}{Spec}

\DeclareMathOperator{\vol}{vol}

\DeclareMathOperator{\Hilb}{Hilb}
\DeclareMathOperator{\Grass}{Grass}

\DeclareMathOperator{\Tor}{Tor}

\DeclareMathOperator{\supp}{supp}

\newcommand*\ratmap{\mathbin{\tikz [baseline=0ex,-latex, dashed, ->] \draw [densely dashed] (0em,0.58ex) -- (1.3em,0.58ex);}}

\DeclarePairedDelimiter\abs{\lvert}{\rvert}

\DeclarePairedDelimiter\floor{\lfloor}{\rfloor}

\newcommand\HH{\mathrm{H}}

\usepackage{color}

\definecolor{orange}{rgb}{1,0.5,0}

\RequirePackage[normalem]{ulem}

\title[Maehara--Severi for maps of general type]{The Theorem of Maehara--Severi for maps of general type}
 
\author{Finn Bartsch}
\address{Finn Bartsch \\
IMAPP Radboud University Nijmegen \\
PO Box 9010, 6500GL \\
Nijmegen, The Netherlands}
\email{f.bartsch@math.ru.nl}

\author{Ariyan Javanpeykar}
\address{Ariyan Javanpeykar \\ 
IMAPP Radboud University Nijmegen \\
PO Box 9010, 6500GL\\
Nijmegen, The Netherlands}
\email{ariyan.javanpeykar@ru.nl}

\author{Erwan Rousseau}
\address{Erwan Rousseau \\
Univ Brest\\
CNRS UMR 6205 \\
Laboratoire de Mathematiques de Bretagne Atlantique\\ F-29200 Brest, France}
\email{erwan.rousseau@univ-brest.fr}



\begin{document}

\begin{abstract}  
We prove a finiteness result for dominant rational maps whose orbifold base is of general type.
Our finiteness result generalizes Maehara's theorem that a given variety dominates only finitely many projective varieties of general type up to birational equivalence, and also answers a question of Campana on the finiteness of Bogomolov sheaves.
We give several further applications, including finiteness results for maps to curves, abelian varieties, and K3 surfaces.
\end{abstract}
\maketitle

\tableofcontents

\thispagestyle{empty}

\section{Introduction}
In 1926, Severi proved that for a variety $Y$ over an algebraically closed field $k$ of characteristic zero, the set of isomorphism classes of smooth projective curves $C$ of genus at least two dominated by $Y$ is finite; see \cite{Samuel}.
In 1983, Severi's theorem was generalized to higher-dimensional varieties by Maehara \cite{Maehara} after work of Martin-Deschamps and Lewin-M\'en\'egaux \cite{DeschampsLewinMenegaux}, see also \cite{GuerraPirola1, GuerraPirola2}.
In this paper we extend these finiteness results to Campana's maps of general type.
Whereas Maehara's theorem concerns maps to varieties of general type, our version concerns maps whose orbifold base is of general type, a substantially broader condition.
Conceptually, our result shows that finiteness is controlled by the orbifold structure induced by the map rather than by the birational type of the target itself.
\smallbreak

To state Maehara's theorem, let $X$ be a variety.
We say that two rational maps $f \colon X \ratmap Y$ and $f' \colon X \ratmap Y'$ are \emph{equivalent} if there is a birational map $\psi \colon Y' \ratmap Y$ such that $\psi \circ f' = f$.
We say that a proper variety $Y$ is of \emph{general type} if there is a resolution of singularities $Y' \to Y$ with $Y'$ a smooth projective variety and $\omega_{Y'}$ big. 
 
\begin{theorem*}[Maehara]
Let $X$ be a smooth variety. 
Then the set of equivalence classes of dominant rational maps $f \colon X \ratmap Y$ with $Y$ a proper variety of general type is finite.
\end{theorem*}
 
Maehara proved his finiteness theorem contingent on a boundedness result (see Theorem \ref{thm:hacon} for a precise statement) which was only proven later after contributions by Tsuji \cite{Tsuji}, Takayama \cite{Takayama1} and Hacon--McKernan \cite{HaconMcKernan}. \smallbreak

Maehara's theorem has found broad applications in algebraic geometry. For example, it is used in the characterization of complex algebraic surfaces definable over a number field \cite[Theorem~7]{GonzalezDiez}, in the construction of Albanese exotic varieties \cite{SchoenChad}, in the proof of the non-isotriviality of certain families of varieties of general type \cite[Lemma~3.2]{Abramovich}, and it forms the starting point for investigations on images of products of (general) curves \cite{BastianelliPirola, GuerraPirola3, LeePirola}.
Finally, let us also mention Hwang and Mok's analogue of Maehara's theorem for Fano varieties \cite{HwangMok}. \smallbreak

To state our generalization of Maehara's result, we follow Campana \cite[Definition~4.2]{Ca11} and define the orbifold base associated to a dominant morphism of smooth varieties, which encodes the nowhere reduced fibers (also called ``inf-multiple fibers'') of this morphism in terms of a $\mathbb{Q}$-divisor on the target. 
 
Let $f \colon X \to Y$ be a dominant morphism of smooth varieties over $k$.
For a prime divisor $D \subseteq Y$, the pullback $f^* D$ is a (possibly empty) divisor which we may decompose as 
\[ f^* D = R + \sum_i a_i F_i, \]
where no irreducible component of $R$ dominates $D$, the $F_i$ are prime divisors dominating $D$ and $a_i\geq 1$. 
We define the \emph{inf-multiplicity} of $f$ over $D$ to be $m_f(D) = \inf(a_i) \in \mathbb{Z}_{\geq 1}\cup \{\infty\}$ (where we adopt the convention that the infimum of the empty set is $\infty$).
Note that the fiber of $f$ over the generic point of $D$ is nowhere reduced if and only if $m_f(D) > 1$. 
We define the \emph{orbifold base divisor} $\Delta_f$ \emph{of $f$} to be the $\mathbb{Q}$-divisor 
\[ \Delta_f := \sum_{D\subset Y} \left(1-\frac{1}{m_f(D)}\right) D, \]
where the sum runs over all prime divisors $D$ of $Y$ (and where we adopt the convention that $\frac{1}{\infty} = 0$).
We refer to the pair $(Y,\Delta_f)$, or simply to the divisor $\Delta_f$, as the \emph{orbifold base of}~$f$.
It is not hard to show that if $g \colon X' \to X$ is a proper birational morphism with $X'$ smooth, then $\Delta_f = \Delta_{f\circ g}$ (cf. Lemma~\ref{orbifold_base_wd}).
Thus, we can define the orbifold base of a dominant rational map $f \colon X \ratmap Y$ by $\Delta_f := \Delta_{f'}$ with $f' \colon X' \to Y$ a resolution of indeterminacy.

We now extend Campana's notion of a fibration of general type \cite[Definition~1.3.1]{Ca04} to arbitrary dominant rational maps.
(We note that Campana's definition of the Kodaira dimension of a fibration also involves modifying $X$; this is unnecessary by the above.)

\begin{definition} \label{def:kodaira_dim_of_morphism}
Let $f \colon X \ratmap Y$ be a dominant rational map of smooth (not necessarily proper) varieties. We define the \emph{Kodaira dimension of $f$} to be 
\[ \kappa(f) = \inf_{f'\colon X\dashrightarrow Y'} \kappa(Y', K_{Y'} + \Delta_{f'}), \]
where $\kappa$ denotes the Iitaka dimension and the infimum is taken over all dominant rational maps $f' \colon X \ratmap Y'$ with $Y'$ smooth proper which are equivalent to $f$.
We say that $f$ is \emph{of general type} if $\kappa(f) = \dim Y$.\footnote{The notion of a morphism of general type as defined by Brunebarbe \cite{BrunebarbeGLL} (being a morphism whose generic fiber is of general type) is unrelated to our definition.}
\end{definition}

We note that the Kodaira dimension of a rational map $X \ratmap Y$ can increase when replacing $X$ by a dense open.
It is however invariant under proper birational modifications $X' \to X$ by Lemma~\ref{orbifold_base_wd}.
We also note that while a priori, Definition~\ref{def:kodaira_dim_of_morphism} requires one to compute the Iitaka dimension of infinitely many divisors defined on infinitely many smooth proper models of $Y$, we show in Remark~\ref{compute_kodaira_dim_on_neat_model} how to compute $\kappa(f)$ using only one model of the map $f$.\smallbreak

The Kodaira dimension of a dominant morphism $f \colon X \to Y$ can be interpreted as measuring the number of pluridifferential forms (perhaps with poles) $\omega$ on $Y$ whose pullback $f^*\omega$ is regular on $X$ with at most logarithmic poles at infinity; see Lemma~\ref{gen_type_criterion} for a precise statement.
\smallbreak

Campana's original motivation for introducing maps of general type (or fibrations of general type) is his notion of a \emph{special} variety \cite[Definition~2.1]{Ca04}, where we recall that a smooth variety $X$ is special if there is no dominant rational map $X\ratmap Y$ of general type with $Y$ positive-dimensional.\smallbreak

The following examples illustrate that maps of general type naturally arise even when the target is not of general type.

If $X$ is a smooth projective variety of general type and $G$ is a finite group acting faithfully on $X$ such that the quotient $Y := X/G$ is smooth, then the quotient map $\pi \colon X \to Y$ is of general type (whereas the target $Y$ need not be of general type).
Indeed, it follows directly from the definition that the orbifold base of $\pi$ is given as $\Delta = \sum_{i} (1-\frac{1}{m_i})B_i$, where the $B_i$ run over the irreducible components of the branch locus and $m_i$ is the order of the stabilizer over $B_i$.
By Riemann--Hurwitz, we have that $\pi^*(K_Y + \Delta) = K_X$; in particular the bigness of $K_X$ implies that $K_Y + \Delta$ is big as well.
Since $Y$ is smooth and $\pi$ is finite surjective, $\pi$ is flat, hence neat (Definition~\ref{def:neat}), so that the infimum in Definition~\ref{def:kodaira_dim_of_morphism} is in fact realized on the model $X \to Y$ (Remark~\ref{compute_kodaira_dim_on_neat_model}).
Thus, $X \to Y$ is of general type.\smallbreak

Another source of examples arises from elliptic fibrations with divisible fibres: if $\pi \colon X \to \mathbb{P}^1$ is an elliptic fibration with at least $5$ divisible fibers, then $\pi$ is of general type (in this case, the classical notion of the ``orbifold base of an elliptic fibration'' coincides with the definition given above). \smallbreak

Motivated mainly by applications to rational points on varieties over function fields, and also by questions related to the finiteness of Bogomolov sheaves on varieties (Corollary~\ref{cor:campana_question}), we prove the following generalization of the theorems of Severi and Maehara to maps of general type.

\begin{alphtheorem}\label{thm:severi_orbifold}
Let $X$ be a smooth variety over an algebraically closed field $k$ of characteristic zero.
Then the set of equivalence classes of dominant rational maps $f\colon X \ratmap Y$ with $f$ of general type is finite.
\end{alphtheorem}


Maehara's theorem complements the finiteness result of Kobayashi and Ochiai that for fixed proper varieties $X$ and $Y$ over an algebraically closed field $k$ of characteristic zero, the set of dominant rational maps $X \ratmap Y$ is finite whenever $Y$ is of general type \cite{KobaOchiai}; this result was extended to varieties of log-general type by Tsushima \cite{Tsushima}. In a similar vein, our Theorem \ref{thm:severi_orbifold} complements the orbifold version of Kobayashi--Ochiai's theorem established in \cite{BJ} (at least in characteristic zero). \smallbreak

Theorem \ref{thm:severi_orbifold} implies Campana's extension of the De Franchis theorem to one-dimensional smooth proper C-pairs of general type \cite[\S3]{CampanaMultiple}; see Section~\ref{section:cpairs} for a detailed discussion. \smallbreak

Even when $X$ is projective, Theorem \ref{thm:severi_orbifold} gives a stronger statement than Maehara’s result.
For instance, it shows that any such $X$ admits only finitely many fibrations $X \to \mathbb{P}^1$ with at least five nowhere reduced fibers (up to automorphisms of $\mathbb{P}^1$), since the orbifold base in this case is of general type. 
In the same way, a smooth projective curve $X$ admits only finitely many totally ramified morphisms to $\mathbb{P}^1$ with at least five branch points (up to automorphisms of $\mathbb{P}^1$).\smallbreak 

Theorem \ref{thm:severi_orbifold} has interesting consequences for maps to abelian varieties and K3 surfaces.
For example, it implies that, given a smooth variety $X$, there are only finitely many simple abelian varieties $A$ such that $X$ admits a dominant morphism $X\to A$ with a nowhere reduced fiber in codimension one.
In particular, since the empty scheme is nowhere reduced, the set of pairs $(A,D)$ with $A$ a simple abelian variety and $D\subseteq A$ a nontrivial effective divisor such that $X$ admits a dominant morphism to $A\setminus D$ is finite up to automorphisms of $A$.
Similarly, the set of isomorphism classes of K3 surfaces $S$ of Picard rank one such that $X$ admits a dominant rational map $X\ratmap S$ with a nowhere reduced fiber in codimension one is finite.\smallbreak

Theorem~\ref{thm:severi_orbifold} applies in particular to fibrations of general type, that is, maps of general type with geometrically connected generic fiber, and thereby answers a question of Campana concerning Bogomolov sheaves \cite[Question~3.15]{Ca11}.
To explain this question, we first recall Campana's definition of a Bogomolov sheaf; such sheaves encode fibrations of general type on a fixed pair $(X,D)$ (see \cite[\S 9.2]{Ca11}).\smallbreak

Let $X$ be a smooth proper variety and let $D$ be a simple normal crossings divisor on $X$.
Let $\Omega^1_X(\log D)$ be the sheaf of differential forms with at most logarithmic poles along $D$, and write $\Omega^p_X(\log D) = \bigwedge^p \Omega^1_X(\log D)$. 
Recall that any saturated sub-line bundle $\mathcal{L} \subseteq \Omega^p_X(\log D)$ satisfies $\kappa(\mathcal{L}) \leq p$; see \cite[\S~12, Thm.~4]{BogomolovHol}, \cite[Cor.~6.9]{EVbook}, or the more general treatment in \cite{Graf2}. 
For $1 \leq p \leq \dim X$, a \emph{Bogomolov sheaf of rank $p$ on the pair $(X,D)$} is a saturated sub-line bundle $\mathcal{L} \subseteq \Omega^p_X(\log D)$ whose Iitaka dimension satisfies $\kappa(\mathcal{L}) = p$.
(This coincides with Campana's definition as explained in \cite[Remark~2.3]{BJL}.)
A \emph{Bogomolov sheaf on $(X,D)$} is a Bogomolov sheaf of some rank $p\geq 1$ on $(X,D)$.\smallbreak 

Given a fibration of general type $X \setminus D \ratmap Y$ with $Y$ of dimension $p$, the saturation in $\Omega^p_X(\log D)$ of the pullback of $\omega_Y$ is a Bogomolov sheaf of rank $p$ on $(X,D)$.
Campana showed that this gives a one-to-one correspondence between the set of Bogomolov sheaves on $(X,D)$ and the set of equivalence classes of fibrations of general type with source $X\setminus D$; see \cite[Th\'eor\`eme~9.9]{Ca11}.\smallbreak

Motivated by Maehara's theorem for varieties of general type, Campana asked whether, for a fixed pair $(X,D)$, the set of Bogomolov sheaves on $(X,D)$ is finite; see \cite[Question~3.15]{Ca11}.
This finiteness is a direct consequence of Theorem~\ref{thm:severi_orbifold}.

\begin{alphcorollary}[Campana's question]\label{cor:campana_question}
Let $X$ be a smooth proper variety and let $D$ be a simple normal crossings divisor on $X$.
Then the set of Bogomolov sheaves for $(X,D)$ is finite.
\end{alphcorollary}

We emphasize that Theorem~\ref{thm:severi_orbifold} is strictly stronger than Corollary~\ref{cor:campana_question}, since it applies to all maps of general type and not only to fibrations.
The proof of Theorem~\ref{thm:severi_orbifold} begins with a finiteness result for \emph{generically finite} maps and then derives the general case, including fibrations, as a consequence; see Section~\ref{section:proof_of_main_theorem}.\smallbreak

Our version of Maehara's theorem also implies the analogous ``log-version''. To state this, recall that a smooth variety $Y$ is of \emph{log-general type} if there is a smooth proper compactification $Y\subseteq \overline{Y}$ whose boundary $D := \overline{Y}\setminus Y$ is an snc divisor such that $K_{\overline{Y}}+D$ is big (i.e., the identity $Y\to Y$ is of general type).
Moreover, recall that a rational map $X \ratmap Y$ is said to be \emph{proper-rational} if there is a proper birational morphism $X'\to X$ such that $X'\ratmap Y$ is a proper morphism (this terminology was introduced by Iitaka \cite[\S~2.12]{Iitaka}).
Moreover, two varieties $X$ and $X'$ are \emph{proper-birational} if there is a proper-rational map $X' \ratmap X$ whose inverse is also proper-rational.
Two rational maps $f\colon X\ratmap Y$ and $f'\colon X\ratmap Y'$ are \emph{proper-birationally equivalent} if there is a proper-birational map $\psi\colon Y\ratmap Y'$ such that $\psi\circ f = f'$.
The quasi-projective version of Maehara's theorem now reads as follows (see Section~\ref{section:log} for its proof). 

\begin{alphcorollary} \label{cor:log_severi}
Let $X$ be a smooth variety over an algebraically closed field $k$ of characteristic zero. 
Then the set of proper-birational equivalence classes of dominant proper-rational maps $X \ratmap Y$ with $Y$ a smooth variety of log-general type is finite.
\end{alphcorollary}
 
Finally, in Section~\ref{section:cpairs} we explain how our methods relate to Campana’s theory of C-pairs (also referred to as ``orbifolds'') and why we formulate our results in terms of dominant rational maps of general type.
In particular, we show that the one-dimensional case of Theorem \ref{thm:severi_orbifold} recovers Campana’s extension of the De Franchis theorem and implies a C-pair version of Severi's classical finiteness result for algebraic curves.
The higher-dimensional situation is more subtle, and we include a short discussion of the issues surrounding C-pair targets for completeness.

\subsection{Conventions} 
We work over an algebraically closed field $k$ of characteristic zero.
A \emph{variety} is an integral separated scheme of finite type over $k$.
If $X$ and $Y$ are varieties, we write $X \times Y$ for $X \times_{\Spec k} Y$. 
If $\mathcal{L}$ is a line bundle, $D$ is a $\mathbb{Q}$-divisor, and $n$ is a natural number, we abuse notation and write $(\mathcal{L}(D))^{\otimes n}$ for $\mathcal{L}^{\otimes n}(\floor{nD})$. 

\begin{ack}
A.J.\ was supported by the Netherlands Organization for Scientific Research NWO, through grant OCENW.M.23.206.
A.J.\ gratefully acknowledges support from the Laboratoire de Math\'ematiques de Bretagne Atlantique (LMBA) and a CNRS Poste Rouge.
\end{ack}

\section{Preliminaries}

\subsection{Neat maps}

We follow Campana \cite[Definition~1.2]{Ca04} and make the following definition. 
(Given a morphism $f\colon X\to Y$ of varieties and a prime divisor $D$ on $X$, we say that $D$ is \emph{contracted by $f$} if $f(D)$ has codimension at least two in $Y$.)

\begin{definition} \label{def:neat}
A dominant morphism $f \colon X \to Y$ between smooth varieties is \emph{neat} if there is a proper birational morphism $p \colon X \to X'$ with $X'$ also smooth such that every prime divisor contracted by $f$ is also contracted by $p$.
Suppose that additionally, $X$ and $Y$ admit morphisms to a smooth variety $T$ and $f$ is a $T$-morphism.
Then we say that $f$ is \emph{relatively neat over $T$} if $X'$ can be chosen to admit a morphism to $T$ such that $p$ is a $T$-morphism.
\end{definition}

If $T = \Spec k$ is a point, being relatively neat is equivalent to being neat.
Since a flat morphism does not contract any divisors, any flat morphism is neat.
Also, every birational morphism of smooth varieties is neat.
Raynaud--Gruson's flattening theorem \cite{RaynaudGruson} ensures that every dominant rational map admits a neat model (see also \cite[Lemma~1.3]{Ca04}).

\begin{lemma} \label{lemma:raynaud_gruson}
Let $X$, $Y$ and $T$ be smooth varieties and let $X \to T$ and $Y \to T$ be morphisms.
Let $f \colon X \to Y$ be a dominant morphism over $T$.
Then there exist projective birational surjective morphisms $X'' \to X$ and $Y' \to Y$ with $X''$ and $Y'$ smooth such that $f$ can be lifted to a relatively neat morphism $X'' \to Y'$ over $T$ whose orbifold base is an snc divisor.
Furthermore, we can ensure that for every $t$ in $T$ such that $X_t''$ and $Y_t'$ are smooth varieties, the induced morphism $X_t'' \to Y_t'$ is neat as well.
\end{lemma}
\begin{proof}
By Raynaud--Gruson's flattening theorem \cite[Theorem~5.2.2]{RaynaudGruson}, there is a proper birational surjective morphism $Y' \to Y$ such that the main component $X'$ of $X \times_Y Y'$ is flat over $Y'$.
Since formation of the main component commutes with further blowups and flatness is preserved under base change, this $Y'$ can be chosen to be smooth and such that the orbifold base of $X' \to Y'$ is snc.
Now let $X''\to X'$ be a proper birational surjective morphism with $X''$ smooth.
Then every divisor contracted by the composed morphism $X'' \to X' \to Y'$ must already be contracted by $X'' \to X'$ since $X' \to Y'$ is flat.
In particular, every such divisor gets contracted by $X'' \to X$, showing that $X'' \to Y'$ is neat.
The same argument also applies to the morphisms $X_t'' \to Y_t'$ for every $t$ in $T$, and thus shows that these are neat as well.
\end{proof}

The importance of neat maps for us comes from Lemma~\ref{lemma:pullback_of_forms} below.

\subsection{The orbifold base}

The following is well-known and can be deduced from \cite[Lemme~4.4]{Ca11} after unraveling the definitions in \emph{loc.~cit.}; we include a short proof for the reader's convenience. 

\begin{lemma}\label{orbifold_base_wd}
Let $f \colon X \to Y$ be a dominant morphism of smooth varieties and let $g \colon X' \to X$ be a proper birational morphism with $X'$ smooth.
Then $\Delta_{f} = \Delta_{f \circ g}$.
\end{lemma}
\begin{proof}
This may be checked one prime divisor at a time, so let $D \subseteq Y$ be a prime divisor.
Write $f^*D = \sum_{i \in I} a_i F_i + R$ with the $F_i \subseteq X$ being prime divisors which dominate $D$ under $f$ and $R$ consisting of prime divisors which get contracted under $f$.
Fix one index $i \in I$ and consider the pullback $g^* F_i$.
Then $g^* F_i = \widetilde{F_i} + \sum_{j \in J_i} b_{i,j} E_{i,j}$ with the $E_{i,j}$ being exceptional divisors for $g$, the $b_{i,j}$ being positive integers, and $\widetilde{F_i}$ denoting the strict transform of $F_i$.
Thus, we see that $\widetilde{F_i}$ appears in $g^*f^*D$ with multiplicity $a_i$ and moreover, that every $E_{i,j}$ appears with multiplicity at least $a_i b_{i,j} \geq a_i$.
In particular, the minimal $a_i$ has not changed and the conclusion follows.
\end{proof}

\begin{remark}
In \cite[Lemma~1.9]{Ca04}, a similar statement to Lemma~\ref{orbifold_base_wd} is claimed in which $X$ and $Y$ are only assumed to be normal instead of smooth.
However, Lemma~\ref{orbifold_base_wd} is wrong if we do not require $X$ to be at least locally factorial -- in the proof above, this assumption is implicitly used when considering the pullback $g^* F_i$, which requires $F_i$ to be Cartier.
For a counterexample, let $X \subseteq \mathbb{A}^3$ be the cone $x^2 = yz$ and consider the morphism $f \colon X \to \mathbb{A}^1$ given by $(x,y,z) \mapsto y$.
Then the fiber over $y = 0$ is the double line $x^2 = 0$, and thus, we would expect $\Delta_f = \frac{1}{2}[0]$.
However, if we consider the blowup $g \colon X' \to X$ in the cone tip, which resolves the singularity, an easy local calculation shows that the exceptional divisor $E \subseteq X'$ gets mapped to $0$ under $f$ and appears with multiplicity $1$ in the fiber $f^{-1}(0)$.
Thus, we have $\Delta_{f \circ g} = 0$ and the analogue of Lemma~\ref{orbifold_base_wd} fails to hold.
It is for this reason that we only consider the orbifold base for morphisms between smooth varieties in what follows. 
\end{remark}

\subsection{Pullback of differential forms} \label{sect:diff_forms}

Let $X$ be a variety over $k$, let $p \colon X' \to X$ be a resolution of singularities, and let $\overline{X'}$ be a compactification of $X'$ such that the boundary $\overline{X'} \setminus X'$ is a simple normal crossings divisor $D$.
In this situation, for any integer $d \leq \dim(X)$ we can consider the sheaf $\Omega^d_{\overline{X'}}(\log D)$ and its tensor powers $\Omega^d_{\overline{X'}}(\log D)^{\otimes r}$.
A key property of these sheaves is that their space of global sections $\HH^0(\overline{X'}, \Omega^d_{\overline{X'}}(\log D)^{\otimes r})$ only depends on $X$, i.e.\ it is independent of the choice of resolution and compactification \cite[§11.1]{Iitaka}.
We call this space the space of \emph{differential forms on $X$ with at most log-poles at infinity}.
If $f \colon X \to Y$ is a dominant morphism of varieties, it is easy to see that pullback of differentials preserves the property of having at most log-poles at infinity.
In fact, we have the following more general result, which also partially explains the relevance of the orbifold base in our setting.

\begin{lemma} \label{lemma:pullback_of_forms2}
Let $X$ and $Y$ be smooth varieties and let $D \subseteq X$ be an snc divisor.
Let $f \colon X \to Y$ be a dominant morphism such that the induced morphism $X \setminus D \to Y$ is neat and let $\Delta_{f,D}$ be the orbifold base of $X \setminus D \to Y$.
Assume that $\Delta_{f,D}$ is supported on an snc divisor.
Then, pullback of differential forms yields an injection $\HH^0(\omega_Y(\Delta_{f,D})^{\otimes r}) \to \HH^0(\Omega^d_X(\log D)^{\otimes r})$, with $d := \dim Y$.
\end{lemma}
\begin{proof}
(This is similar to \cite[Corollary~1.24]{Ca04}.)

Let $\omega \in \HH^0(\omega_Y(\Delta_{f,D})^{\otimes r})$ be a pluridifferential form.
First, note that since the coefficients of $\Delta_{f,D}$ are all at most $1$ and $\Delta_{f,D}$ has snc support, the form $\omega$ has at most a logarithmic pole along any divisor in $Y$.
Thus, the pullback $f^*\omega$ has at most logarithmic poles.
Consequently, it suffices to check that $f^*\omega$ is a regular form on $X \setminus D$ so that from now on, we may assume $D = 0$.

To verify that $f^*\omega$ is regular, let $E \subseteq X$ be any divisor.
Assume first that $f$ does not contract $E$ and hence maps it to a divisor $f_*E$ in $Y$.
In this case, let $m_E$ be the coefficient of $E$ in $f^*f_*E$ and note that $m_E$ is at least the multiplicity $m$ of $f_*E$ in $\Delta_{f,D}$.
An easy local computation around the generic point of $E$ now shows that $f^*\omega$ does not have a pole along $E$.
Consequently, $f^*\omega$ cannot have a pole along any divisor not contracted by $f$.
Since $f$ is neat, there is a proper birational morphism $p \colon X \to X'$ such that every divisor contracted by $f$ is also contracted by $p$.
That is, every divisor on $X'$ corresponds (via the strict transform) to a divisor on $X$ not contracted by $f$.
In particular, viewing $f^*\omega$ as a differential form on $X'$, we see that it does not have any poles in codimension one, and thus, by smoothness of $X'$, must be regular.
It follows that $f^*\omega$ is a regular differential form on $X$ as well.
\end{proof}

We note that in the above proof, we only used that $\Delta_{f,D}$ has snc support when verifying that $f^*\omega$ has at most logarithmic poles.
In the special case $D = 0$, the assumption on the support of $\Delta_{f,D}$ can thus be dropped.

The special case of Lemma~\ref{lemma:pullback_of_forms2} where $X$ and $Y$ have the same dimension will be particularly important to us, so we record it separately.

\begin{lemma} \label{lemma:pullback_of_forms}
Let $X$ and $Y$ be smooth varieties of the same dimension and let $D \subseteq X$ be an snc divisor.
Let $f \colon X \to Y$ be a dominant morphism such that the induced morphism $X \setminus D \to Y$ is neat and let $\Delta_{f,D}$ be the orbifold base of $X \setminus D \to Y$.
Assume that $\Delta_{f,D}$ is supported on an snc divisor.
Then, pullback of differential forms yields an injection $\HH^0(\omega_Y(\Delta_{f,D})^{\otimes r}) \to \HH^0(\omega_X(D)^{\otimes r})$. 
\end{lemma}

We will also need the following relative version of this statement.
(If $f \colon X \to S$ is a dominant morphism of smooth varieties over $k$, we denote $\omega_{X/S} := \omega_X \otimes f^* \omega_S^{\vee}$.)

\begin{lemma} \label{lemma:pullback_of_forms_relative}
Let $T$ be a smooth variety over $k$.
Let $p \colon \mathcal{X} \to T$ and $q \colon \mathcal{Y} \to T$ be dominant morphisms such that $\mathcal{X}$ and $\mathcal{Y}$ are smooth varieties of the same dimension.
Let $F \colon \mathcal{X} \to \mathcal{Y}$ be a dominant morphism over $T$.
Let $D \subseteq \mathcal{X}$ be an snc divisor and let $\Delta_{F,D}$ be the orbifold base of the induced morphism of varieties $\mathcal{X} \setminus D \to \mathcal{Y}$.
Assume that $\mathcal{X} \setminus D \to \mathcal{Y}$ is relatively neat over $T$ and that $\Delta_{F,D}$ is supported on an snc divisor.
Then, pullback of differentials induces an injective sheaf morphism $q_* (\omega_{\mathcal{Y}/T}(\Delta_{F,D})^{\otimes r}) \to p_*(\omega_{\mathcal{X}/T}(D)^{\otimes r})$.
\end{lemma}
\begin{proof}
By Lemma~\ref{lemma:pullback_of_forms}, pullback of differentials gives an injective map
\[ \HH^0(\omega_{\mathcal{Y}}(\Delta_{F,D})^{\otimes r}) \to \HH^0(\omega_{\mathcal{X}}(D)^{\otimes r}). \]
Observe that if $U \subseteq T$ is a dense open, then the restricted morphism $(\mathcal{X} \setminus D)|_U \to \mathcal{Y}|_U$ is relatively neat over $U$, with orbifold base $(\Delta_{F,D})|_U$.
Thus, the assumptions of the lemma are preserved under replacing $T$ with a dense open and $\mathcal{X}$ and $\mathcal{Y}$ with their preimages under $F$.
Consequently, the maps on $\HH^0$ above yield an injective morphism of sheaves
\[ q_*(\omega_{\mathcal{Y}}(\Delta_{F,D})^{\otimes r}) \to p_*(\omega_{\mathcal{X}}(D)^{\otimes r}). \]
Tensoring with $\omega_T^{\otimes -r}$ on both sides, we obtain an injective morphism of sheaves
\[ q_*(\omega_{\mathcal{Y}}(\Delta_{F,D})^{\otimes r}) \otimes \omega_T^{\otimes -r} \to p_* (\omega_{\mathcal{X}}(D)^{\otimes r}) \otimes \omega_T^{\otimes -r}. \]
The sheaf $\omega_T$ is a line bundle and hence, the projection formula applies.
By definition, $\omega_{\mathcal{X}/T} = \omega_{\mathcal{X}} \otimes p^*\omega_T^{\otimes -1}$ and similarly for $\omega_{\mathcal{Y}/T}$; thus, we obtain an injective morphism of sheaves
\[ q_* (\omega_{\mathcal{Y}/T}(\Delta_{F,D})^{\otimes r}) \to p_*(\omega_{\mathcal{X}/T}(D)^{\otimes r}). \]
This is the claim.
\end{proof}

\subsection{Characterizing maps of general type}

In this section, we give a ``model-independent'' criterion for a surjective morphism $f \colon X \to Y$ to be of general type, i.e.\ a criterion which does not require modifying $X$ and $Y$ in its birational equivalence class.
Our starting point is the observation that Lemma~\ref{lemma:pullback_of_forms} has the following converse, even in the absence of any neatness conditions.

\begin{lemma} \label{lemma:pullback_of_forms_converse}
Let $f \colon X \to Y$ be a surjective morphism of smooth varieties and let $D \subseteq X$ be an snc divisor.
Let $r$ be an integer and let $\omega$ be a rational section of $\omega_Y^{\otimes r}$ (i.e.\ a pluridifferential $d$-form on $Y$ which may have poles).
Suppose that $f^* \omega$ is a regular pluridifferential $d$-form on $X \setminus D$ which has at most log-poles along $D$.
Then $\omega$ is a regular section of $\omega_Y(\Delta_{f,D})^{\otimes r}$, where $\Delta_{f,D}$ denotes the orbifold base of $X \setminus D \to Y$.
\end{lemma}
\begin{proof}
Let $E \subseteq Y$ be a prime divisor along which $\omega$ has a pole of order $n \geq 1$ and let $m$ be the multiplicity of $E$ in $\Delta_f$.
We have to show that $n \leq r(1-\frac{1}{m})$.
If $m \neq \infty$, let $E' \subseteq X$ be a non-$f$-exceptional divisor occuring in $f^*E$ with coefficient $m$, such that $E' \nsubseteq \supp D$.
Otherwise, let $E' \subseteq X$ be a component of $D$ mapping to $E$.
To compute the pole order of $f^*\omega$ along $E'$, we may pass to the local ring $\mathcal{O}_{X,E'}$ and even the completed local ring $\widehat{\mathcal{O}}_{X,E'}$.
Thus, consider the induced morphism of complete discrete valuation rings $\widehat{\mathcal{O}}_{Y,E} \to \widehat{\mathcal{O}}_{X,E'}$ and let $t_E$, $t_{E'}$ be the respective uniformizers, chosen such that the image of $t_E$ is $t_{E'}^{m'}$ for some $m' \geq 1$.
Of course, if $m \neq \infty$, we have $m' = m$.
The form $\omega$ can be written as
\[ \omega = t_E^{-n} g (dt_E \wedge dx_1 \wedge \ldots \wedge dx_{d-1})^{\otimes r}, \]
where $g,x_1,\ldots,x_{d-1} \in \widehat{\mathcal{O}}_{Y,E}$ are units.
Computing the pullback $f^* \omega$ then yields
\[ f^* \omega = t_{E'}^{-nm'+(m'-1)r} g (dt_{E'} \wedge dx_1 \wedge \ldots \wedge dx_{d-1})^{\otimes r}. \]
If $m < \infty$, our assumption was that $f^*\omega$ is regular.
Thus, we have $-nm'+(m'-1)r \geq 0$, which immediately implies the desired $n \leq r(1-\frac{1}{m})$.
Otherwise, the assumption that $f^* \omega$ has at most log-poles implies $-nm' + (m'-1)r \geq -r$, which gives $n \leq r$, as required.
\end{proof}

As a consequence, we obtain the following result which relates the global sections of $\omega_Y(\Delta_f)^{\otimes r}$ to a set which only depends on the equivalence class of $f$.
 
\begin{proposition} \label{orbifold_base_vs_differentials_on_base}
Let $f \colon X \to Y$ be a surjective morphism of smooth varieties and let $D \subseteq X$ be an snc divisor.
Write $d = \dim Y$ and let $T_r$ be the space of $r$-pluridifferential $d$-forms on $Y$ (possibly with poles, i.e.\ rational sections of $\omega_Y^{\otimes r}$) which, when pulled back along $f$, become regular forms on $X \setminus D$ with at most log-poles along $D$.
Let $\Delta_{f,D}$ be the orbifold base of $X \setminus D \to Y$.
Then the following hold.
\begin{enumerate}
\item The space $T_r$ is invariant under replacing $X$ by a smooth proper modification $X'$ for which the preimage $D'$ of $D$ in $X'$ is snc.
\item The space $T_r$ is invariant under replacing $Y$ by a smooth proper modification $Y'$ and $X$ by a resolution $X'$ of the main component of $X \times_{Y} Y'$ for which the preimage of $D'$ of $D$ in $X'$ is snc.
\item For any $f$, we have $T_r \subseteq \HH^0(Y, \omega_Y(\Delta_{f,D})^{\otimes r})$.
\item If $X \setminus D \to Y$ is neat and $\Delta_{f,D}$ is supported on an snc divisor, equality holds in (iii).
\end{enumerate}
\end{proposition}
\begin{proof}
Point (i) is just a reformulation of the fact that the global sections of $\Omega^d_X(\log D)^{\otimes r}$ and $\Omega^d_{X'}(\log D')^{\otimes r}$ are the same.
This, combined with the fact that the rational sections of $\omega_Y^{\otimes r}$ and $\omega_{Y'}^{\otimes r}$ are the same, then also implies (ii).
Point (iii) follows immediately from Lemma~\ref{lemma:pullback_of_forms_converse}.
Lastly, (iv) is a consequence of Lemma~\ref{lemma:pullback_of_forms2}, combined with (iii).
\end{proof}

As a direct consequence, we obtain the following alternative characterization of the Kodaira dimension of a morphism (Definition~\ref{def:kodaira_dim_of_morphism}).
Note the lack of any neatness assumptions on $f$.

\begin{corollary} \label{gen_type_criterion}
In the situation of Proposition~\ref{orbifold_base_vs_differentials_on_base}, the morphism $X \setminus D \to Y$ has Kodaira dimension $\kappa$ if and only if $r \mapsto \dim T_r$ grows like a polynomial of degree $\kappa$ in $r$ as $r \to \infty$.
\end{corollary}
\begin{proof}
By Proposition~\ref{orbifold_base_vs_differentials_on_base}.(iii), the growth rate of the dimensions $\dim T_r$ is always \emph{at most} the Iitaka dimension of $K_Y + \Delta_{f,D}$.
By Proposition~\ref{orbifold_base_vs_differentials_on_base}.(i) and .(ii), the $\dim T_r$ do not change when replacing $f$ by a morphism $f'$ equivalent to it.
Thus, varying over all morphisms $f'$ equivalent to $f$, we see that the $\dim T_r$ grow at most like a polynomial of degree $\kappa$ in $r$.

On the other hand, by Lemma~\ref{lemma:raynaud_gruson}, the morphism $f$ is equivalent to a dominant neat morphism $f'$ whose orbifold base is supported on an snc divisor.
By Proposition~\ref{orbifold_base_vs_differentials_on_base}.(i) and .(ii), we may compute $\dim T_r$ using $f'$ instead of $f$.
Since by definition $f$ and $f'$ have the same Kodaira dimension, replacing $f$ by $f'$, we may assume that $f$ is neat and that $\Delta_{f,D}$ has snc support.
In that situation, by Proposition~\ref{orbifold_base_vs_differentials_on_base}.(iv), we have that $T_r = \HH^0(Y,\omega_Y(\Delta_{f,D})^{\otimes r})$.
In particular, $\dim T_r$ grows like a polynomial of degree $\kappa(Y, K_Y + \Delta_{f,D})$.
If $X \setminus D \to Y$ has Kodaira dimension $\kappa$, the Iitaka dimension of $K_Y + \Delta_{f,D}$ is at least $\kappa$, and so, the dimension $\dim T_r$ grows at least like a polynomial of degree $\kappa$ in $r$.
\end{proof}

\begin{remark} \label{compute_kodaira_dim_on_neat_model}
Corollary~\ref{gen_type_criterion}, combined with Proposition~\ref{orbifold_base_vs_differentials_on_base}, shows that if $f \colon X \ratmap Y$ is a rational map and $X' \to X$, $Y' \to Y$ are proper birational morphisms such that $f$ lifts to a \emph{neat} morphism $f' \colon X' \to Y'$ with orbifold base $\Delta$ supported on an snc divisor, we have that $\kappa(f) = \kappa(Y', K_{Y'} + \Delta)$.
In other words, the morphism $f'$ realizes the infimum in Definition~\ref{def:kodaira_dim_of_morphism}.
(Similar results have been noted before, cf. \cite[Proposition~1.25]{Ca04} or \cite[Corollaire~5.11]{Ca11}.)
\end{remark}

\begin{remark}\label{remark:Tr}
Since the spaces $T_r$ are invariant under suitable modifications of $X$ and $Y$, it also makes sense to define them for any dominant rational map $f \colon X \ratmap Y$ of (not necessarily proper) varieties $X$ and $Y$.
To do so, let $X'\to X$ be a proper modification with $X'$ smooth, $\overline{X'}$ an snc compactification of $X'$ with snc boundary $D$, $Y'\to Y$ a proper modification with $Y'$ smooth, and $\overline{Y'}$ a smooth compactification of $Y'$ such that $f$ lifts to a morphism $\overline{X'}\to \overline{Y'}$.
Then define $T_r$ (of $f$) using the morphism $\overline{X'}\to \overline{Y'}$ and the divisor $D$; by Proposition \ref{orbifold_base_vs_differentials_on_base}.(i) and .(ii), this is independent of the choices made.
The characterization of $\kappa(f)$ given in Corollary~\ref{gen_type_criterion} continues to hold, since $\kappa(f)$ is, by definition, also invariant under proper birational modifications of $X$ and arbitrary birational modifications of $Y$.
\end{remark}

Lastly, we obtain the following result, which as a special case implies that if one restricts a morphism of general type with positive-dimensional fibers to a hyperplane section, one again obtains a morphism of general type.
(This result is similar to \cite[Proposition~2.10]{Ca04}).

\begin{corollary} \label{precomposition_stability}
Let $f \colon Y \ratmap Z$ be a dominant rational map of general type with $Y$ and $Z$ smooth.
Let $X$ be a smooth variety and let $g \colon X \to Y$ be a morphism such that the composition $f \circ g \colon X \ratmap Z$ is defined and dominant.
Then $X \ratmap Z$ is of general type.
\end{corollary}
\begin{proof}
By Corollary~\ref{gen_type_criterion} and Remark~\ref{remark:Tr}, it suffices to prove that the dimensions of the spaces of pluridifferential forms on $Z$ whose pullback along $f\circ g$ is regular on $X$ with at most log-poles at infinity grow like a polynomial of degree $\dim(Z)$.
To show this, first observe that if $\omega$ is a pluridifferential form on $Z$ whose pullback to $Y$ is regular with at most log-poles at infinity, then its pullback to $X$ is regular with at most log-poles at infinity as well.
Consequently, since $f$ is of general type, the claim follows by Corollary~\ref{gen_type_criterion}.
\end{proof}

\subsection{The orbifold base in families}

In this section, we prove that given a family of morphisms over some base variety $T$, the formation of the orbifold base commutes with restriction to the fiber for general $t \in T$.
A similar (but weaker) statement was proven independently by Cadorel--Deng--Yamanoi \cite[Lemma~4.16]{CadorelDengYamanoi}.
Let us first make our setup precise.

Let $T$ be a smooth variety over $k$ and let $\mathcal{X} \to T$, $\mathcal{Y} \to T$ be smooth surjective morphisms with $\mathcal{X}$ and $\mathcal{Y}$ integral varieties of the same dimension.
Assume that $\mathcal{Y} \to T$ has geometrically connected fibers.
Let $F \colon \mathcal{X} \to \mathcal{Y}$ be a dominant morphism over $T$ and let $\Delta_F$ be its orbifold base.

\begin{lemma}\label{lemma:orbifold_base_restricted_to_fibers}
In the above setting, there is a dense open subset $T^\circ \subseteq T$ such that for every $t \in T^\circ$, the restricted map $F_t \colon \mathcal{X}_t \to \mathcal{Y}_t$ is dominant and its orbifold base satisfies $\Delta_{F_t} = (\Delta_F)|_{\mathcal{Y}_t}$.
\end{lemma}
\begin{proof}
First, note that there is a dense open $T^\circ \subseteq T$ such that $F_t$ is dominant for every $t \in T^\circ$.
(Indeed, the set $T^\circ$ may be taken to be the image in $T$ of the interior of the image of $F$.)
Thus, shrinking $T$ if necessary, we may assume that every $F_t$ is dominant.

Since $F$ is a dominant morphism between varieties of the same dimension, there is a closed subset $Z \subseteq \mathcal{Y}$ of codimension at least two such that the induced map $F^{-1}(\mathcal{Y} \setminus Z) \to \mathcal{Y} \setminus Z$ has finite fibers.
Moreover, since the morphisms $F_t$ are all dominant, the intersection of $Z$ with any $\mathcal{Y}_t$ has codimension at least two in $\mathcal{Y}_t$.
Thus, replacing $\mathcal{Y}$ by $\mathcal{Y} \setminus Z$ changes neither the orbifold base of $F$ nor the orbifold base of the $F_t$.
Consequently, we may assume that $F$ has finite fibers and in particular, that there are no $F$-exceptional divisors and no $F_t$-exceptional divisors for any $t \in T$.

Consider the maximal open set $\mathcal{X}^\circ$ on which $F$ is étale and let $\mathcal{Y}^{\circ} := F(\mathcal{X}^\circ)$.
Since étale maps are open, $\mathcal{Y}^{\circ}$ is open in $\mathcal{Y}$.
Replacing $T$ by a dense open if necessary, we may assume that $\mathcal{Y}^{\circ} \to T$ is still surjective.

Observe that if $E \subseteq \mathcal{Y}$ is a divisor not contained in $\Delta_F$, the pullback $F^*E$ contains a divisor with multiplicity $1$.
Thus, $\mathcal{X}^\circ$ has nonempty intersection with $F^{-1}(E)$ and the generic point of $E$ is contained in $\mathcal{Y}^\circ$.
Conversely, if $E \subseteq \supp \Delta_F$, then, at a general point of $\mathcal{X}$ lying over $E$, the map $F$ is not étale, and hence the generic point of $E$ does not lie in $\mathcal{Y}^\circ$.
It follows that $\mathcal{Y} \setminus \mathcal{Y}^\circ = \supp \Delta_F \cup Z$ for some closed subset $Z \subseteq \mathcal{Y}$ of codimension at least two.
Shrinking $T$ if necessary, we may assume that for every $t \in T$, we have that $Z_t \subseteq \mathcal{Y}_t$ is of codimension at least two as well.

As the morphism $\mathcal{X}^{\circ} \to \mathcal{Y}^{\circ}$ is étale surjective, we see that for every $t$ in $T$ and every $y \in \mathcal{Y}^\circ_t$, there is at least one point of $\mathcal{X}_t$ lying over $y$ at which $F_t$ is étale.
Such $y$ cannot be contained in the orbifold base of $F_t$, hence we see that for every $t \in T$, we have $\supp \Delta_{F_t} \subseteq \supp (\Delta_F|_{\mathcal{Y}_t})$.
It remains to make sure that the multiplicities agree.

To ensure this, we must shrink $T$ further, so that the following conditions are satisfied for all $t \in T$:
\begin{enumerate}[label=(\roman*)]
\item For every irreducible component $E \subseteq \supp \Delta_F$, the intersection $E \cap \mathcal{Y}_t$ is reduced.
\item For every irreducible component $E \subseteq \supp \Delta_F$ and every irreducible component $D \subseteq F^{-1}(E)$, the intersection $D \cap \mathcal{X}_t$ is reduced. (Here, $F^{-1}(E)$ denotes the set-theoretic preimage.)
\item For every irreducible component $E \subseteq \supp \Delta_F$ and every irreducible component $D \subseteq F^{-1}(E)$, the degree of the map $D \cap \mathcal{X}_t \to E \cap \mathcal{Y}_t$ is the same as the degree of $D \to E$. (In particular, every irreducible component of $E \cap \mathcal{Y}_t$ has a divisor lying over it.)
\item For every pair of distinct irreducible components $E_1, E_2 \subseteq \supp \Delta_F$, the intersections $E_1 \cap \mathcal{Y}_t$ and $E_2 \cap \mathcal{Y}_t$ share no irreducible component of codimension one in $\mathcal{Y}_t$.
\end{enumerate}

Now, fix a $t \in T$ and an irreducible divisor $E \subseteq \supp (\Delta_F|_{\mathcal{Y}_t})$.
Then, there is an irreducible component $\widetilde{E} \subseteq \Delta_F$ (unique by condition (iv) above) such that $\mathcal{Y}_t \cap \widetilde{E}$ contains $E$ as an irreducible component.
By condition (i) above, $E$ appears in this intersection with coefficient $1$.
Let $m$ be the multiplicity of $\widetilde{E}$ in $\Delta_F$ and let $\widetilde{D}$ be a component of $F^* \widetilde{E}$ whose coefficient in $F^* \widetilde{E}$ is $m$.
By condition (iii) above, there is an irreducible component $D$ of $\widetilde{D} \cap \mathcal{X}_t$ mapping to $E$, which, by condition (ii), occurs in $\widetilde{D} \cap \mathcal{X}_t$ with coefficient $1$.
Thus, $D$ occurs in $F_t^*E$ with coefficient $m$.
This shows that the multiplicity of $E$ in $\Delta_{F_t}$ is at most $m$.
Hence $\Delta_{F_t} \leq (\Delta_F)|_{\mathcal{Y}_t}$.

On the other hand, we have that $F_t^*E \subseteq \mathcal{X}_t \cap F^*\widetilde{E}$, i.e.\ every divisor $D$ appearing in $F_t^*E$ is the restriction to $\mathcal{X}_t$ of a divisor $\widetilde{D}$ appearing in $F^*\widetilde{E}$, and the coefficients agree.
In particular, every coefficient of a divisor appearing in $F_t^*E$ also appears in $F^*\widetilde{E}$, so that the multiplicity of $\widetilde{E}$ in $\Delta_F$ is at most the multiplicity of $E$ in $\Delta_{F_t}$.
This shows the other inequality, finishing the proof.
\end{proof}

\section{Bounds on graphs of dominant maps}

In this section we prove that the family of maps of general type with fixed source $X$ is bounded (Theorem \ref{thm:boundedness}); this extends Maehara's result \cite[Proposition~3.3]{Maehara}.
In our proof, we follow Maehara's strategy and start with the following uniformity statement.
(Recall that a $\mathbb{Q}$-divisor $\Delta$ on a variety is said to have \emph{standard coefficients} if its coefficients lie in the set $\{ 1-\frac{1}{m}~|~m \in \mathbb{Z}_{\geq 1} \} \cup \{1\}$.)

\begin{theorem} \label{thm:hacon}
Fix a positive integer $d$. 
Then, there is a positive integer $r$ such that for every $d$-dimensional smooth projective variety $X$ and every effective $\mathbb{Q}$-divisor $\Delta$ on $X$ with standard coefficients and snc support such that $K_X + \Delta$ is big, the divisor $\floor{r(K_X + \Delta)}$ defines a map 
\[ X \ratmap \abs{\floor{r(K_X + \Delta)}}^\vee \]
which is birational onto its image.
\end{theorem}
\begin{proof}
The set of standard coefficients satisfies the descending chain condition (abbreviated by DCC in \cite{HaconMcKernanXu}), so that the theorem follows from \cite[Theorem~C]{HaconMcKernanXu}.
\end{proof}

\begin{remark}
If $d=1$, using the classification of smooth proper orbifold curves of general type, one can show that $r = 42$ satisfies the above conclusion.
\end{remark}

\begin{remark}[Volume and degree]\label{remark:volume}
Let $X$ be a proper variety of dimension $d$ and let $D \subseteq X$ be a divisor.
Then the \emph{volume} of $D$, denoted by $\vol_X(D)$, is defined by
\[ \vol_X(D) := \lim_{m \to \infty} \frac{d!}{m^d} h^0(mD). \]
If $D \leq E$ are two divisors, it is immediate from the definition that $\vol_X(D) \leq \vol_X(E)$.
For nef divisors $D$, we have the easier formula $\vol_X(D) = D^d$; see \cite[Corollary~1.4.41]{Lazarsfeld1}.
In particular, if $D$ is very ample and $X \subseteq \mathbb{P}^n$ is the associated embedding, the volume of $D$ is the degree of $X$ in $\mathbb{P}^n$.


If $p \colon X' \to X$ is a proper birational map and $D$ is a divisor on $X$, then $\vol_{X'}(p^*D) = \vol_X(D)$; see \cite[Proposition~2.2.43]{Lazarsfeld1}.
Thus, if $D$ is a big divisor such that the rational map $X \ratmap \mathbb{P}^n$ is birational onto its image $X_0$, we see that $\deg(X_0) \leq \vol_X(D)$.
Moreover, if $f \colon X \to Y$ is a generically finite morphism of proper varieties, there is an inequality $\vol_Y(D) \leq \vol_X(f^*D)$.
Hence we see that if $D$ is a big divisor on a proper variety $X$ such that the associated rational map $X \ratmap \mathbb{P}^n$ is generically finite onto its image $X_0$, we still have $\deg(X_0) \leq \vol_X(D)$.

The main use of the volume for us lies in the following observation:
Let $f \colon X \ratmap Y$ be a rational map of proper varieties, let $H$ be an ample divisor on $X$ and let $D$ be a big divisor on $Y$ such that the induced rational map $Y \ratmap \mathbb{P}^n$ is birational onto its image $Y_0$.
Consider the graph $\Gamma \subseteq X \times Y_0 \subseteq X \times \mathbb{P}^n$ of $f$.
Then, the degree of $\Gamma$ with respect to the ample class $(H, \mathcal{O}_{\mathbb{P}^n}(1))$ is bounded above by $\vol_X(H+f^*D)$.
Indeed, this is essentially a direct consequence of the previous discussion, as $X$ is birational to $\Gamma$.
\end{remark}

\begin{theorem}\label{thm:boundedness}
Let $X$ be a smooth variety over $k$.
Then, there is an integer $n \geq 1$ such that there exist finite type $k$-schemes $H_1,H_2,\ldots,H_n$ as well as closed subschemes $\mathcal{F}_m \subseteq H_m \times (X\times \mathbb{P}^m)$ for $m=1,\ldots,n$ satisfying the following property:
For every dominant rational map $f \colon X \ratmap Y$ of general type with $\dim X = \dim Y$, there exists a closed point $h$ in some $H_m$ such that $\mathcal{F}_{m,h} \subseteq X \times \mathbb{P}^m$ is the graph of a rational map $f_h \colon X \ratmap \mathbb{P}^m$ whose image factorization is equivalent to $f$. 
\end{theorem}
\begin{proof}
Let $d = \dim(X)$ and let $r$ be as in Theorem \ref{thm:hacon}. 
Consider the space $V := \HH^0(\omega_X^{\otimes r})_{\log}$ of differential forms which have at most log-poles at infinity (defined in Section \ref{sect:diff_forms}), and let $n =\dim V - 1$.

Consider a dominant rational map $f \colon X \ratmap Y$ of general type with $\dim X = \dim Y$.
Choose a neat model $f' \colon X' \to Y'$ of $f$ whose orbifold base $\Delta_{f'}$ is snc (Lemma~\ref{lemma:raynaud_gruson}).
As the rational map $f \colon X \ratmap Y$ is of general type, the divisor $K_{Y'} + \Delta_{f'}$ on $Y'$ is big.
Thus, our choice of $r$ ensures that the rational map $Y' \ratmap \mathbb{P}^m$ induced by the linear system $\abs{\floor{r(K_{Y'} + \Delta_{f'})}}$ is birational onto its image.
Moreover, by Lemma~\ref{lemma:pullback_of_forms}, there is a natural injection $\HH^0(Y',\omega_{Y'}(\Delta_{f'})^{\otimes r}) \to \HH^0(X', \omega_{X'}^{\otimes r})_{\log}$ and by the discussion in Section~\ref{sect:diff_forms}, the latter space identifies with $V$.
In particular, we have $m \leq n$.

Fix a compactification $X \subseteq \overline{X}$ with snc boundary $D$ and an ample divisor class $H$ on $\overline{X}$.
Let $\Gamma \subseteq \overline{X} \times \mathbb{P}^m$ be the graph of the rational map $f \colon X \ratmap Y$ considered above, composed with the birational embedding $Y \ratmap \mathbb{P}^m$.
By Remark~\ref{remark:volume}, we see that the degree of $\Gamma$ with respect to the class $(H, \mathcal{O}_{\mathbb{P}^m}(1))$ is bounded by $\vol_{\overline{X}}(H + f^*(\floor{r(K_{Y'} + \Delta_{f'})}))$.
In particular, using Lemma~\ref{lemma:pullback_of_forms} again, we see that the degree of $\Gamma$ is bounded by $\vol_{\overline{X}}(H + K_{\overline{X}} + D)$.
Observe that this bound is independent of $f$ and instead, only depends on $X$, $\overline{X}$ and $H$.

Now, as in \cite[Section~2]{Maehara}, a result due to Chow implies that all the graphs $\Gamma$ lie in finitely many irreducible components of the Hilbert scheme $\Hilb(\overline{X} \times \mathbb{P}^m)$, where $m$ runs over the integers from $1$ to $n$. 
Thus, we can take each $H_m$ to be the collection of the relevant components of $\Hilb(\overline{X} \times \mathbb{P}^m)$ and the $\mathcal{F}_m$ to be the associated universal families.
\end{proof}

\section{Weak positivity and nefness}

We will need the following weak positivity result for the sheaf of relative pluricanonical forms. 
Similar versions of this result are due to Kawamata \cite{KawamataAbelian}, Viehweg \cite{Viehweg1,Viehweg2}, Kollár \cite[Corollary~3.7]{Kollar86}, Campana \cite[Theorem~4.13]{Ca04}, Nakayama \cite[Theorem~3.35]{Nakayama}, Fujino \cite[Theorem~1.1]{Fujino} \cite[Theorem~1.7]{Fujino18}, and Dutta--Murayama \cite[Theorem~D]{DuttaMurayama}.

Recall that a locally free sheaf $\mathcal{E}$ on a variety $X$ is said to be \emph{nef} if the line bundle $\mathcal{O}(1)$ on $\mathbb{P}(\mathcal{E})$ is nef.

\begin{theorem} \label{thm:campana}
Let $C$ be a smooth projective curve, let $X$ be a smooth projective variety and let $f \colon X \to C$ be a projective morphism with geometrically connected fibers.  
Let $\Delta$ be a $\mathbb{Q}$-divisor on $X$ and suppose that there is a dense open $U \subseteq C$ such that the restriction of $\Delta$ to $f^{-1}(U)$ has snc support and has coefficients in $[0,1]$.
Then, for every integer $m$, the locally free sheaf $f_*(\omega_{X/C}(\Delta)^{\otimes m})$ on $C$ is nef.
\end{theorem}
\begin{proof}
Since $C$ is a smooth curve, every torsionfree sheaf on $C$ is locally free, so that $f_*(\omega_{X/C}(\Delta)^{\otimes m})$ is locally free.
Replacing $\Delta$ by $\frac{1}{m} \floor{m \Delta}$, we may assume that $m \Delta$ is a $\mathbb{Z}$-divisor.

Let $p \colon X' \to X$ be a proper birational morphism which is an isomorphism over $f^{-1}(U)$ such that the union of the strict transform of $\supp \Delta$ and the exceptional divisors of $p$ is an snc divisor on $X'$.
Then we can write
\begin{equation}\label{discrepancy_equation}\tag{$*$}
K_{X'} + \Delta' + E' = p^*(K_X + \Delta) + E, 
\end{equation}
where $\Delta'$, $E$ and $E'$ are effective $\mathbb{Q}$-divisors on $X'$ which do not share any irreducible components and with $\Delta'$ being the strict transform of the horizontal part of $\Delta$.
Note that $E$ is $p$-exceptional and that $E'$ does not dominate $C$ under $f \circ p$.
By construction, the divisor $\Delta'$ has snc support and has coefficients in $[0,1]$.
Thus, the pair $(X', \Delta')$ is log-canonical, so that it follows from \cite[Theorem~D]{DuttaMurayama} that $(f \circ p)_*(\omega_{X'/C}(\Delta')^{\otimes m})$ is weakly positive on $C$.  
Consequently, since a weakly positive sheaf on a curve is automatically nef, it is nef.  

As $E'$ is effective, we have an inclusion $\omega_{X'/C}(\Delta')^{\otimes m} \subseteq \omega_{X'/C}(\Delta'+E')^{\otimes m}$, which induces an inclusion $(f \circ p)_*(\omega_{X'/C}(\Delta')^{\otimes m}) \subseteq (f \circ p)_*(\omega_{X'/C}(\Delta'+E')^{\otimes m})$. 
Since $E'$ does not dominate $C$, these sheaves are generically isomorphic and thus have the same rank.
In particular, since $(f \circ p)_*(\omega_{X'/C}(\Delta')^{\otimes m})$ is nef, the sheaf $(f \circ p)_*(\omega_{X'/C}(\Delta'+E')^{\otimes m})$ is nef as well.  

Now, by Equation (\ref{discrepancy_equation}), we have $\omega_{X'/C}(\Delta'+E')^{\otimes m} \cong p^*\omega_{X/C}(\Delta)^{\otimes m} \otimes \mathcal{O}_{X'}(mE)$.  
Thus, the locally free sheaf $(f \circ p)_*\left(p^*(\omega_{X/C}(\Delta)^{\otimes m}) \otimes \mathcal{O}_{X'}(mE)\right)$ is nef.
Since $E$ is effective and $p$-exceptional, we have that $p_*\mathcal{O}_{X'}(mE) \cong \mathcal{O}_X$.
Hence the projection formula implies that $p_*(p^*(\omega_{X/C}(\Delta)^{\otimes m}) \otimes \mathcal{O}_{X'}(mE)) \cong \omega_{X/C}(\Delta)^{\otimes m} \otimes p_*\mathcal{O}_{X'}(mE) \cong \omega_{X/C}(\Delta)^{\otimes m}$.
Consequently, $f_*(\omega_{X/C}(\Delta)^{\otimes m})$ is nef, as desired.  
\end{proof}

\section{Triviality of the relative Iitaka map}

In this section we show that, under suitable regularity assumptions, the relative Iitaka map associated to a family of dominant maps with fixed source is trivial; see Corollary~\ref{cor:iitaka} for a precise statement.
This is roughly equivalent to showing that, given a family $(f_t \colon X \to Y_t)_{t \in T}$ of such morphisms, the image of the pullback map $\HH^0(Y_t, \omega_{Y_t}(\Delta_t)^{\otimes N}) \to \HH^0(X,\omega_{X}^{\otimes N})$ is independent of $t$.
To do so, we will use the Grassmannian associated to the vector space $\HH^0(X,\omega_{X}^{\otimes N})$.

\subsection{The Grassmannian}

Let $V$ be a finite dimensional $k$-vector space and $\nu$ an integer.
Let $\Grass_\nu(V)$ be the Grassmannian parametrizing quotient spaces of $V$ of dimension $\nu$.
Recall that $\Grass_\nu(V)$ is smooth and projective over $k$ of dimension $\nu(\dim V -\nu)$ and that for a $k$-scheme $T$, we have
\[ \Grass_\nu(V)(T) = \left\{ 
\begin{gathered} \text{isomorphism classes of surjective}~\mathcal{O}_T\text{-module homomorphisms} \\ q \colon V \otimes_{k} \mathcal{O}_T \to \mathcal{Q}~\text{with}~\mathcal{Q}~\text{locally free of rank}~\nu \end{gathered}
\right\}. \]

From the functorial interpretation of the Grassmannian, we see that the identity $\Grass_\nu(V) \to \Grass_\nu(V)$ corresponds to a morphism of $\mathcal{O}_{\Grass_\nu(V)}$-modules $q_{\text{univ}} \colon V \otimes_{k} \mathcal{O}_{\Grass_\nu(V)} \to \mathcal{Q}_{\text{univ}}$ with $\mathcal{Q}_{\text{univ}}$ locally free of rank $\nu$.
The bundle $\mathcal{Q}_{\text{univ}}$ is known as the \emph{universal quotient bundle}.
It also follows immediately that given a morphism $f \colon T \to \Grass_\nu(V)$ and its corresponding $\mathcal{O}_T$-module homomorphism $q \colon V \otimes_k \mathcal{O}_T \to \mathcal{Q}$, we have that $f^*\mathcal{Q}_{\text{univ}} \cong \mathcal{Q}$ and moreover $f^*q_{\text{univ}} \cong q$.
Furthermore, we obtain a surjection $\left(\bigwedge^\nu V\right) \otimes_{k} \mathcal{O}_{\Grass_\nu(V)} \to \det(\mathcal{Q}_{\text{univ}})$.
By the functorial interpretation of projective space, this surjection gives rise to a morphism $\Grass_\nu(V) \to \mathbb{P}(\bigwedge^\nu V)$ such that the line bundle $\det(\mathcal{Q}_{\text{univ}})$ on $\Grass_\nu(V)$ is the pullback of the $\mathcal{O}(1)$.
This morphism is well-known to be a closed immersion; in particular we see that $\det(\mathcal{Q}_{\text{univ}})$ is ample.

\subsection{Linearization of admissible families}

Let $k$ be an algebraically closed field of characteristic zero.
Throughout this section, we will work in the following setup.

\begin{setup} \label{setup:admissible}
Let $X$ be a smooth projective variety over $k$ and let $D \subseteq X$ be an snc divisor.
Let $N \geq 1$ be an integer.
Let $T$ be a smooth variety, let $\mathcal{X} \to X \times T$ be a projective birational morphism with $\mathcal{X}$ a smooth $k$-variety, and let $q \colon \mathcal{Y} \to T$ be a projective morphism with $\mathcal{Y}$ a smooth $k$-variety of dimension $\dim \mathcal{Y} = \dim \mathcal{X}$.
Let $p \colon \mathcal{X} \to T$ denote the induced morphism and let $\mathcal{D}$ denote the preimage of $D \times T$ in $\mathcal{X}$.
Assume that for every $t \in T$, the divisor $\mathcal{D}_t$ has simple normal crossings.
Let $F \colon \mathcal{X} \to \mathcal{Y}$ be a dominant generically finite morphism over $T$ and let $\Delta$ be the orbifold base of $\mathcal{X} \setminus \mathcal{D} \to \mathcal{Y}$.
Assume that $F$ satisfies the following properties.
\begin{enumerate} 
\item The induced morphism $\mathcal{X} \setminus \mathcal{D} \to \mathcal{Y}$ is relatively neat over $T$.
\item The orbifold base $\Delta$ is supported on an snc divisor.
\item The (torsionfree) coherent sheaf $q_* \left((\omega_{\mathcal{Y}/T}(\Delta))^{\otimes N}\right)$ is locally free.
\item The cokernel of the natural morphism
\[ q_*(\omega_{\mathcal{Y}/T}(\Delta)^{\otimes N}) \to p_*(\omega_{\mathcal{X}/T}(\mathcal{D})^{\otimes N}) \]
is locally free. (This morphism is well-defined by Lemma~\ref{lemma:pullback_of_forms_relative}.)
\end{enumerate}
\end{setup}

Suppose we are in Setup~\ref{setup:admissible}.
We define $V = \HH^0(X,\omega_X(D)^{\otimes N})$; note that $p_*(\omega_{\mathcal{X}/T}(\mathcal{D})^{\otimes N}) = V \otimes \mathcal{O}_T$.
Let $Q$ be the cokernel of
\[ q_\ast(\omega_{\mathcal{Y}/T}(\Delta)^{\otimes N}) \to p_\ast (\omega_{\mathcal{X}/T}(\mathcal{D})^{\otimes N}). \]
By assumption, $Q$ is locally free and we let $\nu$ denote its rank.
Thus, we can make the following definition.

\begin{definition}[Linearization map]
The \emph{linearization of $F$} is the morphism
\[ \mathrm{Lin}_F^N \colon T \to \Grass_\nu(V) \]
defined by $Q$.
\end{definition}

\begin{remark} \label{remark:linearization_map_interpretation}
At points of $T$ over which $F$ is sufficiently well-behaved, the linearization morphism has a straightforward interpretation.
Suppose that $t \in T(k)$ is a point for which the morphism $F_t \colon \mathcal{X}_t\to \mathcal{Y}_t$ is a dominant neat morphism of smooth varieties.
Then we have an injection $F_t^\ast \colon \HH^0(\mathcal{Y}_t, \omega_{\mathcal{Y}_t}(\Delta_t)^{\otimes N}) \to \HH^0(\mathcal{X}_t,\omega_{\mathcal{X}_t}(\mathcal{D}_t)^{\otimes N})$.
Suppose moreover that the natural morphism $\mathcal{X}_t \to X$ is birational and that $\mathcal{D}_t$ is the preimage of $D$ under this morphism.
Then we have a natural isomorphism $\HH^0(\mathcal{X}_t,\omega_{\mathcal{X}_t}(\mathcal{D}_t)^{\otimes N}) \cong \HH^0(X,\omega_X(D)^{\otimes N}) = V$.
Unraveling definitions, we thus see that for such a point $t \in T$, the linearization map $\mathrm{Lin}_F^N$ maps $t$ to the image of $F_t^\ast$ in $V$.
\end{remark}

\begin{theorem}\label{thm:linearization_is_constant}
In the situation of Setup~\ref{setup:admissible}, the linearization map $\mathrm{Lin}_F^N$ is constant.
\end{theorem}
\begin{proof}
It suffices to prove that the restriction of $\mathrm{Lin}_F^N$ to some dense open of $T$ is constant.
Since the formation of the orbifold base is compatible with replacing $T$ by a dense open of $T$, we may thus use Lemma~\ref{lemma:orbifold_base_restricted_to_fibers} and assume that the restriction of the orbifold base $\Delta$ of $F|_{\mathcal{X} \setminus \mathcal{D}}$ to $\mathcal{Y}_t$ is the orbifold base of $F_t|_{\mathcal{X}_t \setminus \mathcal{D}_t}$, for every $t \in T$.
In this situation, we also have that for every smooth quasi-projective curve $C \subseteq T$, the orbifold base of the morphism $F_C \colon \mathcal{X}_C \setminus \mathcal{D}_C \to \mathcal{Y}_C$ is the restriction of $\Delta$ to $\mathcal{Y}_C$.
Consequently, formation of the linearization map commutes with restricting to $C$.
Because a map whose restriction to every smooth quasi-projective curve is constant must be constant itself, we are thus reduced to the case that $T = C$ is a curve.  

Let $\overline{C}$ denote the smooth projective model of $C$.
Before proceeding with the proof, we first want to extend the morphism $F \colon \mathcal{X} \to \mathcal{Y}$ to a morphism $\overline{F} \colon \overline{\mathcal{X}} \to \overline{\mathcal{Y}}$ such that $\overline{\mathcal{X}} \setminus \overline{\mathcal{D}} \to \overline{\mathcal{Y}}$ is relatively neat over $\overline{C}$.
To do so, first choose any projective extension $\overline{q} \colon \overline{\mathcal{Y}}\to \overline{C}$ with $\overline{\mathcal{Y}}$ smooth and any projective birational extension $\overline{\mathcal{X}} \to X \times \overline{C}$ with $\overline{\mathcal{X}}$ smooth such that $F$ extends to a $\overline{C}$-morphism $\overline{F} \colon \overline{\mathcal{X}} \to \overline{\mathcal{Y}}$.
Let $\overline{\mathcal{D}}$ denote the closure of $\mathcal{D}$ in $\overline{\mathcal{X}}$.
Then, using Lemma~\ref{lemma:raynaud_gruson}, birationally modify $\overline{\mathcal{X}}$, $\overline{\mathcal{Y}}$, and $\overline{F}$ to obtain a morphism $\overline{F}' \colon \overline{\mathcal{X}}' \to \overline{\mathcal{Y}}'$ over $\overline{C}$ such that the restricted morphism $\overline{\mathcal{X}}' \setminus \overline{\mathcal{D}}' \to \overline{\mathcal{Y}}'$ is relatively neat over $\overline{C}$ and has an orbifold base $\overline{\Delta}'$ with snc support; here $\overline{\mathcal{D}}'$ denotes the preimage of $\overline{\mathcal{D}}$ in $\overline{\mathcal{X}}'$.
Let $\mathcal{X}'$, $\mathcal{Y}'$, and $F'$ denote the base change of $\overline{\mathcal{X}}'$, $\overline{\mathcal{Y}}'$, and $\overline{F}'$ to the open subset $C \subseteq \overline{C}$.
We want to replace $\mathcal{X}$, $\mathcal{Y}$ and $F$ by $\mathcal{X}'$, $\mathcal{Y}'$ and $F'$; then, since $F'$ extends by construction, we have achieved our goal of being able to extend $F$ in a relatively neat way.

To make sure this step is valid, we need to verify that this does not change the sheaves $p_*(\omega_{\mathcal{X}/C}(\mathcal{D})^{\otimes N})$ and $q_*(\omega_{\mathcal{Y}/C}(\Delta)^{\otimes N})$.
The sections of the sheaf $p_*(\omega_{\mathcal{X}}(\mathcal{D})^{\otimes N})$ have the interpretation of ``pluridifferential forms on $\mathcal{X}$ with at most log-poles along $\mathcal{D}$'', which do not change under proper birational morphisms.
The sections of the sheaf $q_*(\omega_{\mathcal{Y}}(\Delta)^{\otimes N})$ have the interpretation of ``rational pluridifferential forms on $\mathcal{Y}$ whose pullback to $\mathcal{X} \setminus \mathcal{D}$ is regular with at most log-poles along $\mathcal{D}$'' (Proposition~\ref{orbifold_base_vs_differentials_on_base}), and thus also do not change.
Using the isomorphism $p_*(\omega_{\mathcal{X}/C}(\mathcal{D})^{\otimes N}) \cong (p_*(\omega_{\mathcal{X}}(\mathcal{D})^{\otimes N}) \otimes \omega_C^{\otimes -N}$ given by the projection formula, we conclude that the sheaf $p_*(\omega_{\mathcal{X}/C}(\mathcal{D})^{\otimes N})$ stays unchanged, and similarly for $q_*(\omega_{\mathcal{Y}/C}(\Delta)^{\otimes N})$.
We conclude that replacing $\mathcal{X}$, $\mathcal{Y}$, and $F$ by $\mathcal{X}'$, $\mathcal{Y}'$, and $F'$ does not change the linearization map.
Consequently, we may from now on assume that $F$ extends to a morphism $\overline{F} \colon \overline{\mathcal{X}} \to \overline{\mathcal{Y}}$ such that $\overline{\mathcal{X}} \setminus \overline{\mathcal{D}} \to \overline{\mathcal{Y}}$ is relatively neat over $\overline{C}$ and has snc orbifold base $\overline{\Delta}$.
 
Recall that $Q$ denotes the cokernel of the morphism
\[ q_\ast(\omega_{\mathcal{Y}/C}(\Delta)^{\otimes N}) \to p_\ast (\omega_{\mathcal{X}/C}(\mathcal{D})^{\otimes N}). \]
Note that $V\otimes \mathcal{O}_{\overline{C}} = \overline{p}_\ast (\omega_{\overline{\mathcal{X}}/ \overline{C}}(\overline{\mathcal{D}})^{\otimes N})$.
Consequently, Lemma~\ref{lemma:pullback_of_forms_relative} yields a morphism
\[\overline{q}_{\ast}(\omega_{\overline{\mathcal{Y}}/\overline{C}}(\overline{\Delta})^{\otimes N}) \to V \otimes \mathcal{O}_{\overline{C}};\]
let $\overline{Q}$ denote its cokernel.
Let $\Tor(\overline{Q}) \subseteq \overline{Q}$ be the torsion subsheaf, so that $\overline{Q}/\Tor(\overline{Q})$ is locally free on $\overline{C}$.
The surjection $V \otimes \mathcal{O}_{\overline{C}} \to \overline{Q}/\Tor(\overline{Q})$ hence defines a morphism $\overline{C}\to \Grass_\nu(V)$, and since $Q$ is the restriction of $\overline{Q}/\Tor(\overline{Q})$ to $C$, this morphism $\overline{C}\to \Grass_\nu(V)$ uniquely extends $\mathrm{Lin}_F^N\colon C\to \Grass_\nu(V)$.

Let $\mathcal{K}$ be the kernel of the morphism $V \otimes \mathcal{O}_{\overline{C}} \to \overline{Q}/\Tor(\overline{Q})$ and note that $\mathcal{K}$ is a locally free sheaf on $\overline{C}$.
By construction, there is an inclusion $\overline{q}_{\ast}(\omega_{\overline{\mathcal{Y}}/\overline{C}}(\overline{\Delta})^{\otimes N}) \subseteq \mathcal{K}$ which is an isomorphism over $C$.
In particular, $\overline{q}_{\ast}(\omega_{\overline{\mathcal{Y}}/ \overline{C}}(\Delta)^{\otimes N})$ and $\mathcal{K}$ have the same rank.
Theorem~\ref{thm:campana} implies that $\overline{q}_{\ast}(\omega_{\overline{\mathcal{Y}}/\overline{C}}(\overline{\Delta})^{\otimes N})$ is nef on $\overline{C}$.
In particular, $\det \overline{q}_{\ast}(\omega_{\overline{\mathcal{Y}}/\overline{C}}(\overline{\Delta})^{\otimes N})$ is nef, so that
\[ \deg \overline{q}_{\ast}(\omega_{\overline{\mathcal{Y}}/\overline{C}}(\overline{\Delta})^{\otimes N}) \geq 0 \]
and hence
\[ \deg \mathcal{K} \geq \deg \overline{q}_{\ast}(\omega_{\overline{\mathcal{Y}}/ \overline{C}}(\Delta)^{\otimes N}) \geq 0. \]

On the other hand, the pullback of the universal quotient bundle $\mathcal{Q}_{\text{univ}}$ on $\Grass_\nu(V)$ along the morphism $\overline{C}\to \Grass_\nu(V)$ is $\overline{Q}/\Tor(\overline{Q})$.
In particular, since $\det\left( \mathcal{Q}_{\text{univ}}\right)$ is ample, we see that if $\mathrm{Lin}_F^N$ is nonconstant, then $\deg(\overline{Q}/\Tor(\overline{Q}))>0$.
However, $V\otimes \mathcal{O}_{\overline{C}}$ has degree zero, so that $\deg(\overline{Q}/\Tor(\overline{Q})) + \deg(\mathcal{K}) = 0$.
Since $\deg(\mathcal{K}) \geq 0$, we conclude that $\mathrm{Lin}_F^N$ is constant, as desired.
\end{proof}

The fact that $\mathrm{Lin}_F^N$ is constant entails birational triviality of the Iitaka map.
The following statement generalizes \cite[Proposition~4.1.3]{Maehara}; note the lack of a ``general type'' assumption here.

\begin{corollary}\label{cor:iitaka}
In the situation of Setup~\ref{setup:admissible}, assume additionally that $q_{\ast}(\omega_{\mathcal{Y}/T}(\Delta)^{\otimes N})$ is not the zero sheaf. Then, the following statements hold.
\begin{itemize}
\item The rational map $\iota \colon \mathcal{Y} \ratmap \mathbb{P}(q_{\ast}(\omega_{\mathcal{Y}/T}(\Delta)^{\otimes N}))$ is constant over $T$, i.e., there is a variety $Z$ over $k$ such that the image of $\iota$ is $Z \times T$.
\item There is a rational map $\phi\colon X\ratmap Z$ such that the composed rational map $X\times T \ratmap \mathcal{Y}\ratmap Z \times T$ is $\phi\times T$.
\end{itemize}
\end{corollary}
\begin{proof}
Since the sheaf $q_*(\omega_{\mathcal{Y}/T}(\Delta)^{\otimes N})$ is nonzero, the natural morphism of sheaves
\[ q^* q_*(\omega_{\mathcal{Y}/T}(\Delta)^{\otimes N}) \to \omega_{\mathcal{Y}/T}(\Delta)^{\otimes N} \]
is nonzero as well.
Consequently, it is surjective over a dense open of $\mathcal{Y}$ and we indeed obtain a rational map $\mathcal{Y} \ratmap \mathbb{P}(q_*(\omega_{\mathcal{Y}/T}(\Delta)^{\otimes N}))$.

By Lemma~\ref{lemma:pullback_of_forms_relative}, pullback of differentials yields an injective morphism of sheaves
\[ q_*(\omega_{\mathcal{Y}/T}(\Delta)^{\otimes N}) \to p_*(\omega_{\mathcal{X}/T}(D)^{\otimes N}). \]
In particular, the latter sheaf is nonzero and we also have a rational map $\mathcal{X} \ratmap \mathbb{P}(p_*(\omega_{\mathcal{X}/T}(D)^{\otimes N}))$.
Since $\mathcal{X}$ is birational to $X \times T$ over $T$, we have an isomorphism of sheaves $p_*(\omega_{\mathcal{X}/T}(D)^{\otimes N}) \cong V \otimes \mathcal{O}_T$ and the rational map $\mathcal{X} \ratmap \mathbb{P}(p_*(\omega_{\mathcal{X}/T}(D)^{\otimes N}))$ is equivalent to the natural rational map $X \times T \ratmap \mathbb{P}(V) \times T$.
The latter is by construction constant over $T$.
We obtain the following commutative diagram.

\begin{equation*} \begin{tikzcd}
X \times T \ar[r, dashed] \ar[d, dashed] & \mathbb{P}(V) \times T \ar[d, dashed] \\
\mathcal{Y} \ar[r, dashed] & \mathbb{P}(q_*(\omega_{\mathcal{Y}/T}(\Delta)^{\otimes N}))
\end{tikzcd} \end{equation*}

The right vertical map in the above diagram is the projection away from the relative linear subspace $\mathbb{P}(Q)\subseteq \mathbb{P}(V)\times T$, where $Q$ denotes as above the cokernel of  
\[ q_*(\omega_{\mathcal{Y}/T}(\Delta)^{\otimes N}) \to p_*(\omega_{\mathcal{X}/T}(D)^{\otimes N}). \]
Thus, by Theorem~\ref{thm:linearization_is_constant}, the right vertical map is constant over $T$ (i.e., the right vertical map is fiberwise a linear map between projective spaces and its center is independent of $T$). 
It follows that the composition $X \times T \ratmap \mathbb{P}(V) \times T \ratmap \mathbb{P}(q_*(\omega_{\mathcal{Y}/T}(\Delta)^{\otimes N}))$ is constant over $T$ as well.
Since the rational map $X \times T \ratmap \mathcal{Y}$ is dominant, the image of $\mathcal{Y} \ratmap \mathbb{P}(q_*(\omega_{\mathcal{Y}/T}(\Delta)^{\otimes N}))$ coincides with the image of $X \times T \ratmap \mathbb{P}(q_*(\omega_{\mathcal{Y}/T}(\Delta)^{\otimes N}))$.
This immediately implies the assertions.
\end{proof}

\section{Generic rigidity}

We follow Maehara's strategy and prove an extension of his rigidity theorem \cite[Appendix]{Maehara} to maps of general type.
As a first step, we prove the result in the situation where all involved morphisms are neat, the varieties are smooth, and the divisors are snc.
 
\begin{lemma}[Generic rigidity, smooth version] \label{rigidity_smooth}
Let $X$ be a smooth projective variety and let $D$ be an snc divisor on $X$.
Let $T$ be a smooth variety and let $q \colon \mathcal{Y} \to T$ be a smooth projective morphism with geometrically connected fibers.
Let $\mathcal{X} \to X \times T$ be a projective birational morphism such that the induced morphism $p \colon \mathcal{X} \to T$ is smooth and let $\mathcal{D} \subseteq \mathcal{X}$ be the preimage of $D \times T \subseteq X \times T$, viewed as a reduced divisor.
Let $F \colon \mathcal{X} \to \mathcal{Y}$ be a dominant generically finite morphism over $T$.
Assume that the restricted morphism $\mathcal{X} \setminus \mathcal{D} \to \mathcal{Y}$ is relatively neat over $T$ and that its orbifold base $\Delta_{F, \mathcal{D}}$ has snc support.
Assume furthermore that the set of $t \in T(k)$ such that the restriction $F_t \colon (\mathcal{X} \setminus \mathcal{D})_t \to \mathcal{Y}_t$ is a neat morphism of general type is dense in $T$.

Then there are a smooth proper variety $Y$, a dense open $T^\circ$ of $T$ and a rational map $f \colon X \ratmap Y$ such that $\mathcal{Y}|_{T^\circ}$ is birational to $Y \times T^\circ$ and $F$ is equivalent to $f \times T^\circ$ over $T^\circ$.
\end{lemma} 
\begin{proof}
Let $N \geq 1$ be a positive integer.
Since $\omega_{\mathcal{Y}/T}(\Delta_{F, \mathcal{D}})^{\otimes N}$ is coherent and $q$ is proper, the sheaf $q_*\left( (\omega_{\mathcal{Y}/T}(\Delta_{F, \mathcal{D}}))^{\otimes N}\right)$ is coherent.
In particular, replacing $T$ by a dense open if necessary, it is locally free.
Similarly, shrinking $T$ further if necessary, the cokernel of the morphism of coherent sheaves
\[ q_*(\omega_{\mathcal{Y}/T}(\Delta_{F, \mathcal{D}})^{\otimes N}) \to p_* (\omega_{\mathcal{X}/T}(\mathcal{D})^{\otimes N}) \]
(constructed in Lemma~\ref{lemma:pullback_of_forms_relative}) is locally free as well.
Thus, for any choice of $N \geq 1$, there is a dense open $T_N \subseteq T$ such that over $T_N$, we are in the situation of Setup~\ref{setup:admissible}.

If $t \in T$ is a point for which $F_t$ is of general type, the line bundle $\omega_{\mathcal{Y}_t}(\Delta_{F_t, \mathcal{D}_t})^{\otimes r}$ on $\mathcal{Y}_t$ is big.
Thus, by Theorem~\ref{thm:hacon}, there is an integer $r \geq 1$ such that for every $t$ with $F_t$ of general type, the rational map associated to $\omega_{\mathcal{Y}_t}(\Delta_{F_t, \mathcal{D}_t})^{\otimes r}$ is birational onto its image.
Replacing $N$ by a multiple if necessary, we may assume that this holds for $r = N$.
Note that this in particular implies that $\omega_{\mathcal{Y}_t}(\Delta_{F_t, \mathcal{D}_t})^{\otimes N}$ has at least one nonzero global section.

By Lemma \ref{lemma:orbifold_base_restricted_to_fibers}, replacing $T$ by a dense open if necessary, we have that for every $t \in T$, the restriction of $\Delta_{F,\mathcal{D}}$ to $\mathcal{Y}_t$ is exactly $\Delta_{F_t, \mathcal{D}_t}$.
Consequently, Grauert's local freeness theorem \cite[Corollary~III.12.9]{Hartshorne} implies that the fiber of $q_{\ast}(\omega_{\mathcal{Y}/T}(\Delta_{F, \mathcal{D}})^{\otimes N})$ at a point $t \in T$ is the space of global sections of $\omega_{\mathcal{Y}_t}(\Delta_{F_t, \mathcal{D}_t})^{\otimes N}$.
In particular, $q_*(\omega_{\mathcal{Y}/T}(\Delta_{F, \mathcal{D}})^{\otimes N})$ is not the zero sheaf and we obtain a rational map $\mathcal{Y} \ratmap \mathbb{P}(q_*(\omega_{\mathcal{Y}/T}(\Delta_{F, \mathcal{D}})^{\otimes N})))$ over $T$.
For every $t \in T$, this rational map restricts to the map associated to the line bundle $\omega_{\mathcal{Y}_t}(\Delta_{F_t, \mathcal{D}_t})^{\otimes N}$ and thus, its restriction to $\mathcal{Y}_t$ is birational onto its image for a dense set of $t \in T(k)$.
It follows that the rational map $\mathcal{Y} \ratmap \mathbb{P}(q_*(\omega_{\mathcal{Y}/T}(\Delta_{F, \mathcal{D}})^{\otimes N})))$ is birational onto its image.

The result now follows from Corollary~\ref{cor:iitaka}.
\end{proof}

\begin{remark}
In the setting of Lemma \ref{rigidity_smooth}, if $\mathcal{X}\to X\times T$ is assumed to be an isomorphism (i.e., not just birational) and $\mathcal{Y}_t$ is of general type, then the statement can also be deduced from a theorem of Noguchi. In fact, in this case the conclusion can even be strengthened: there is a dense open $T^\circ \subset T$ such that $\mathcal{Y}_{T^\circ}$ is \emph{isomorphic} to $Y\times T^\circ$ over $T^\circ$. This is proven in Noguchi's paper \cite{NoguchiRigidity} (see \cite[Theorem~2.8]{XY1} for details). 
\end{remark}

The general case then follows by reducing it to the nice case outlined above.
(Maehara's original rigidity theorem is Theorem \ref{thm:rigidity} below for $D$ trivial and $\mathcal{Y}\to T$ such that $\mathcal{Y}_t$ is of general type for a dense set of $t$ in $T(k)$.)

\begin{theorem}[Generic rigidity]\label{thm:rigidity}
Let $X$ be a smooth projective variety and let $D$ be an snc divisor on $X$.
Let $T$ be a smooth variety and let $q \colon \mathcal{Y}\to T$ be a smooth projective morphism with geometrically connected fibers.
Let $F \colon (X \setminus D) \times T \ratmap \mathcal{Y}$ be a dominant generically finite rational map over $T$.
Assume that the set of $t\in T(k)$ such that $F_t$ is a map of general type is dense.
Then there are a smooth proper variety $Y$, a dense open $T^\circ$ of $T$ and a rational map $f \colon X \ratmap Y$ such that $\mathcal{Y}|_{T^\circ}$ is birational to $Y \times T^\circ$ and $F$ is equivalent to $f \times T^\circ$ over $T^\circ$.
\end{theorem}
\begin{proof}
Let $\mathcal{X} \to X \times T$ be a proper birational surjective morphism with $\mathcal{X}$ a smooth projective variety such that the induced rational map $\mathcal{X} \ratmap \mathcal{Y}$ is a morphism.
Let $\mathcal{D}$ be the preimage of $D \times T$ in $\mathcal{X}$.
Using Lemma~\ref{lemma:raynaud_gruson} and shrinking $T$ if necessary, we may assume that the morphism $\mathcal{X} \setminus \mathcal{D} \to \mathcal{Y}$ is relatively neat over $T$, that the orbifold base of $\mathcal{X} \setminus \mathcal{D} \to \mathcal{Y}$ has snc support, that $p \colon \mathcal{X} \to T$ and $q \colon \mathcal{Y} \to T$ are smooth projective, and that $\mathcal{X}_t \to \mathcal{Y}_t$ is neat for every $t \in T$.
The induced morphism $\mathcal{X} \setminus \mathcal{D} \to (X \setminus D) \times T$ is proper birational and hence, after possibly shrinking $T$, it is also fiberwise proper birational.
In particular, for any $t \in T$, the morphism $X \setminus D \to \mathcal{Y}_t$ is of general type if and only if $(\mathcal{X} \setminus \mathcal{D})_t \to \mathcal{Y}_t$ is so.
Thus, we may identify the morphism $\mathcal{X} \setminus \mathcal{D} \to \mathcal{Y}$ with $F$ and we denote by $\Delta = \Delta_{F,\mathcal{D}}$ its orbifold base.
The result now follows from Lemma~\ref{rigidity_smooth}.
\end{proof}

\section{The proof of Maehara's theorem for maps of general type}\label{section:proof_of_main_theorem}

We first prove Maehara's theorem in the equidimensional setting.

\begin{lemma}\label{severi_equidimensional_case}
Let $X$ be a smooth variety. Then the set of equivalence classes of dominant maps of general type $X \ratmap Y$ with $\dim Y = \dim X$ is finite.
\end{lemma}
\begin{proof}
By Theorem~\ref{thm:boundedness}, there are finitely many finite type $k$-schemes $H_1, \ldots, H_n$, equipped with closed subschemes $\mathcal{F}_m \subseteq H_m \times (X \times \mathbb{P}^m)$ for $m = 1, \ldots, n$ such that, for every $f \colon X \ratmap Y$ as in the theorem statement, there is a closed point $h$ in some $H_m$ such that $\mathcal{F}_{m,h}$ is the graph of some rational map $X \ratmap \mathbb{P}^m$ whose image factorization is equivalent to $f$.

Now, consider an irreducible component $H \subseteq H_m$ and consider the set of closed points $h \in H$ such that $\mathcal{F}_{m,h}$ is the graph of some dominant rational map $X \ratmap Y$ whose image factorization is of general type.
Either, this set is dense in $H$, in which case Theorem~\ref{thm:rigidity} implies that there is a nonempty, hence dense, open $U \subseteq H$ such that the closed points $h \in U$ all parametrize equivalent rational maps, or it is not dense, in which case we let $U \subseteq H$ be the complement of its closure.
Repeating this argument by applying Theorem~\ref{thm:rigidity} to the irreducible components of $H \setminus U$ and so on, we see that by noetherian induction, the closed points of the $H_i$ only parametrize finitely many equivalence classes of rational maps.
This implies the claim.
\end{proof}

Theorem~\ref{thm:severi_orbifold} now follows from Lemma~\ref{severi_equidimensional_case} by cutting $X$ with hyperplane sections (this is our analogue of Maehara's \cite[Lemma~6.3]{Maehara}).

\begin{proof}[Proof of Theorem \ref{thm:severi_orbifold}]
It suffices to prove this statement for any fixed value of the difference $\dim X - \dim Y$.
We now proceed by induction on $\dim X - \dim Y$.
The base case $\dim X - \dim Y = 0$ is Lemma~\ref{severi_equidimensional_case}.
For the induction step, we argue by contradiction.
So suppose that $(f_i \colon X \ratmap Y_i)_{i=1}^\infty$ is a sequence of pairwise non-equivalent dominant rational maps of general type, where all $Y_i$ have the same dimension $d$.
Replacing $X$ by a dense open subset if necessary, we may assume that $X$ is quasi-projective.
Fix an immersion $X \subseteq \mathbb{P}^n$.
Replacing the ground field by an extension field if necessary, let $H \subseteq X$ be a very general hyperplane section.
Then the $f_i|_H$ are still dominant rational maps and are still pairwise non-equivalent.
Moreover, by Corollary~\ref{precomposition_stability}, they are still of general type.
Since $\dim H - d < \dim X - d$, this contradicts the induction hypothesis, which finishes the proof.
\end{proof}

\subsection{Maehara's theorem for varieties of log-general type} \label{section:log} 
 
Since every dominant morphism to a smooth variety of log-general type is a morphism of general type, Corollary~\ref{cor:log_severi} from the introduction is now an immediate consequence of Theorem~\ref{thm:severi_orbifold} and the following lemma.

\begin{lemma}
Let $X$, $Y_1$, and $Y_2$ be smooth varieties and let $p \colon X \ratmap Y_1$ and $q \colon X \ratmap Y_2$ be two equivalent proper-rational maps.
Then $p$ and $q$ are proper-birationally equivalent.
\end{lemma}
\begin{proof}
By assumption, there is a birational map $\psi \colon Y_1 \ratmap Y_2$ such that $\psi \circ p = q$.
We have to show that $\psi$ is proper-birational.
To do so, let $X' \to X$ be a proper birational morphism such that $p' \colon X' \to X \ratmap Y_1$ and $q' \colon X' \to X \ratmap Y_2$ are proper morphisms.
Then, we can consider the morphism $(p', q') \colon X' \to Y_1 \times Y_2$.
This morphism is proper, so that in particular its image $Y' \subseteq Y_1 \times Y_2$ is closed.
Now, the projection onto the first component $\pi_1 \colon Y' \to Y_1$ is proper, since it is proper after precomposition with the surjective map $X' \to Y'$.
Moreover, if $U \subseteq Y_1$ is an open on which $\psi$ is a morphism, then $\pi_1$ is an isomorphism over $U$, so that $\pi_1$ is proper birational.
Similarly, the projection onto the second component $\pi_2 \colon Y' \to Y_2$ is proper as well, and it is a model for the rational map $\psi$.
Thus, $\psi$ is proper-birational, as desired.
\end{proof}

We stress that the analogous statement to Corollary \ref{cor:log_severi} formulated using strictly rational maps and strict-birational equivalence classes is wrong.
To see this, let $X$ be a minimal smooth projective surface of general type and let $U := X \setminus D$ be the complement of a prime divisor.
Then for every $p \in D$, the inclusion morphism $f_p \colon U \to X \setminus \{p\}$ is a morphism of general type.
However, $f_p$ and $f_q$ are strict-birationally equivalent if and only if $p$ and $q$ are equal up to an automorphism of $X$.
As $\Aut(X)$ is finite, we thus see that the $f_p$ fall into infinitely many strict-birational equivalence classes.
This does not contradict Corollary~\ref{cor:log_severi} since none of the $f_p$ are proper-rational maps.

Note that Corollary \ref{cor:log_severi} appears to be new even for curves.
In the case of curves, however, it can be deduced from the classical theorems of De Franchis and Severi.
As this simple argument (which adapts how one proves ``Faltings implies Siegel'' over number fields \cite[Exercise~E.11]{HindrySilverman}) does not seem to appear in the literature, we include a brief sketch for completeness in the following remark.

\begin{remark} \label{remark:severi_curves} 
Let $X$ be a variety over an algebraically closed field $k$ of characteristic zero. We give a direct proof of the fact that the set of isomorphism classes of smooth curves $Y$ of log-general type admitting a dominant morphism $X \to Y$ is finite.  

Consider smooth curves $Y_1, Y_2, \ldots$ of log-general type together with nonconstant morphisms $f_i \colon X \to Y_i$. For each $i$, choose a finite \'etale Galois morphism $Y_i' \to Y_i$ of degree at most six such that the smooth projective model $\overline{Y_i}'$ of $Y_i'$ has genus at least two. (Such an \'etale cover is easily shown to exist.) Set $X_i = X \times_{Y_i} Y_i'$; then $X_i \to X$ is a finite \'etale morphism of degree at most three. Thus, by the finiteness of \'etale covers of $X$ of bounded degree, we may choose a finite \'etale morphism $X' \to X$ such that every composition $X' \to X \to Y_i$ factors through $Y_i' \to Y_i$. 

By Severi’s classical theorem, since $X'$ dominates every $\overline{Y_i}'$, the set of isomorphism classes of the $\overline{Y_i}'$ is finite. Moreover, by the theorem of De Franchis, for any fixed $i$, the set of nonconstant morphisms $\mathrm{Hom}^{nc}(X',\overline{Y_i}')$ from $X'$ to any $\overline{Y_i}'$ is finite.  For every $g$ in $\mathrm{Hom}^{nc}(X',\overline{Y_i}')$ there are only finitely many dense opens of $\overline{Y_i}'$ containing $\mathrm{Im}(g)$. Since each $Y_i'$ contains the image of some $g$ in $\mathrm{Hom}^{nc}(X',\overline{Y_i}')$, we conclude that the $Y_i'$ are finite in number up to isomorphism. Since each $Y_i$ is a finite group quotient of $Y_i'$ and $\mathrm{Aut}(Y_i')$ is finite (Hurwitz), it follows that the curves $Y_i$ are finite in number up to isomorphism.
\end{remark}

\section{Campana's C-pairs} \label{section:cpairs}  

In this section we explain how Theorem~\ref{thm:severi_orbifold} yields C-pair analogues of the classical theorems of Severi and De Franchis for curves.
We then give a brief discussion of why Maehara’s theorem does not extend naturally to higher-dimensional C-pairs.

We briefly recall the definition of a C-pair and a morphism of C-pairs following \cite{Ca04, Ca11} (see also \cite[\S 1]{BJ}).
Let $X$ be a smooth variety, and let $\Delta = \sum_i (1-\frac{1}{m_i}) D_i$ be a $\mathbb{Q}$-divisor, where $m_i\in \mathbb{Z}_{\geq 1}\cup \{\infty\}$. We refer to $(X,\Delta)$ as a \emph{C-pair}, and say it is \emph{smooth} if $\supp \Delta$ has simple normal crossings, and \emph{proper} if $X$ is proper over $k$. A smooth proper C-pair $(X,\Delta)$ is of \emph{general type} if $K_X+\Delta$ is a big $\mathbb{Q}$-divisor. If $Y$ is a normal variety, a \emph{morphism} of C-pairs $f \colon Y \to (X, \Delta)$ is a morphism of varieties $f \colon Y \to X$ satisfying $f(Y) \nsubseteq \supp \Delta$ such that for every $i$, the coefficients of the pullback $f^*D_i$ are all at least $m_i$.\smallbreak 

The following simple lemma is immediate from the definitions.
\begin{lemma} \label{cpair_morphism_vs_orbifold_base}
Let $X$ be a smooth variety and let $(Y, \Delta_Y)$ be a smooth proper C-pair.
Let $f \colon X \to (Y, \Delta_Y)$ be a dominant morphism and let $\Delta_f$ be the orbifold base of the associated morphism of varieties $X \to Y$.
Then $\Delta_Y \leq \Delta_f$.
\end{lemma} 

\begin{remark}
There are dominant morphisms $f \colon X \to Y$ of smooth varieties with orbifold base $\Delta_f$ such that $X \to (Y, \Delta_f)$ is not a morphism of C-pairs.
Such morphisms are necessarily non-flat and in particular must have $\dim Y \geq 2$; see \cite{Bartsch} for explicit examples.  
\end{remark}

The following result is an immediate consequence of Theorem~\ref{thm:severi_orbifold} and Lemma~\ref{cpair_morphism_vs_orbifold_base}.

\begin{theorem}[Severi for C-pairs]
Let $X$ be a smooth proper variety. 
Then the set of isomorphism classes of smooth proper C-pair curves $(C, \Delta_C)$ of general type admitting a nonconstant C-pair morphism $X \to (C, \Delta_C)$ is finite.
\end{theorem}
\begin{proof}
First note that since we are only considering morphisms into curves, taking the infimum in Definition~\ref{def:kodaira_dim_of_morphism} is irrelevant.
If $f \colon X \to (C, \Delta_C)$ is a nonconstant C-pair morphism with $(C, \Delta_C)$ a C-pair curve of general type and $\Delta_f$ denotes the orbifold base of $X \to C$, we have that $\Delta_f \geq \Delta_C$ (Lemma~\ref{cpair_morphism_vs_orbifold_base}).
Hence $K_C + \Delta_f$ is big and thus, $f$ is a morphism of general type.
As a consequence, by Theorem~\ref{thm:severi_orbifold}, it remains to prove that for any fixed morphism $f \colon X \to C$, there are only finitely many C-pair divisors $\Delta_C$ on $C$ for which $X \to (C, \Delta_C)$ is a C-pair morphism.
Since $X$ is proper, the morphism $X \to C$ is surjective, so that the coefficients of $\Delta_f$ are all strictly less than $1$.
The result now follows from the elementary fact that for any fixed $\mathbb{Q}$-divisor $D$ on any variety $Y$ whose coefficients are strictly less than $1$, there are only finitely many C-pair divisors $\leq D$ on $Y$.
\end{proof}

We can also deduce from Theorem \ref{thm:severi_orbifold} the following analogue of the theorem of De Franchis for C-pairs of dimension one; this result was first proven by Campana \cite[\S3]{CampanaMultiple} and is a special case of \cite[Theorem~1.1]{BJ}.  

\begin{theorem}[De Franchis for C-pairs] \label{de_franchis_for_cpairs}
Let $X$ be a smooth variety and let $(C, \Delta_C)$ be a smooth proper C-pair curve of general type.
Then the set of nonconstant morphisms $X \to (C, \Delta_C)$ is finite.
\end{theorem}
\begin{proof}
By Theorem~\ref{thm:severi_orbifold}, there are only finitely many equivalence classes of nonconstant morphisms $f \colon X \to C$ whose orbifold base $\Delta_f$ is of general type.
Since a nonconstant morphism $X \to (C, \Delta_C)$ must have an orbifold base of general type (Lemma~\ref{cpair_morphism_vs_orbifold_base}), it thus suffices to prove that in any fixed equivalence class of morphisms $X \to C$, only finitely many representatives are C-pair morphisms $X \to (C, \Delta_C)$.
To do so, consider two equivalent nonconstant morphisms $f_1 \colon X \to C$ and $f_2 \colon X \to C$ whose orbifold bases $\Delta_1$ and $\Delta_2$ are both $\geq \Delta_C$ and let $\phi \colon C \to C$ be the automorphism of $C$ such that $\phi \circ f_1 = f_2$.
In this case, we have $\phi^* \Delta_2 = \Delta_1$ and in particular, $\phi$ restricts to an isomorphism between $C \setminus \supp \Delta_1$ and $C \setminus \supp \Delta_2$.
Moreover, this isomorphism gives an isomorphism of $C \setminus \supp \Delta_C$ with a dense open of $C$ containing $C \setminus \supp \Delta_2$.
As there are only finitely many dense opens of $C$ containing $C \setminus \supp \Delta_2$, there are thus only finitely many options for what $\phi(C \setminus \supp \Delta_C)$ can be.
Since the automorphism group of $C \setminus \supp \Delta_C$ is finite, there are also only finitely many choices for what the isomorphism $C \setminus \supp \Delta_C \to \phi(C \setminus \supp \Delta_C)$ can be, and thus only finitely many options for $\phi$.
Thus, fixing $f_2$, there are only finitely many options for $f_1$, which is what we wanted to prove.
\end{proof}

The preceding arguments rely heavily on the fact that the target is a C-pair curve.
The following two remarks explain why analogous finiteness statements  for higher-dimensional C-pairs are less natural, which motivates our use of the orbifold base instead.

\begin{remark}[On birational equivalence for C-pairs]
There is no satisfactory notion of birational equivalence for C-pairs.
Given a birational morphism $\mu \colon Y' \to Y$ of smooth varieties and a C-pair structure $(Y, \Delta_Y)$, it is in general unclear how to define natural multiplicities on the $\mu$-exceptional divisors so as to obtain a birational transform $(Y', \Delta_{Y'})$ that is unique or functorial in any reasonable sense.
This issue has been observed repeatedly (see for example Chen--Lehmann--Tanimoto's discussion of Campana weak approximation over function fields \cite[Remark~5.5]{CLT}). 
Consequently, even \emph{formulating} finiteness statements ``up to birational equivalence of C-pairs'' is problematic.
\end{remark}

\begin{remark}[On finiteness of C-pair targets]\label{rem:finiteness_targets_c_pairs}
Let $f \colon X \to Y$ be a dominant morphism between smooth varieties, and let $\Delta_f$ be its orbifold base.
There are typically many C-pair structures $(Y, \Delta_Y)$ for which $f$ becomes a morphism of C-pairs: as long as $f$ is flat, any C-pair divisor $\Delta_Y \leq \Delta_f$ suffices.
If one allows non-surjective morphisms $X \to Y$, then already a single fixed rational map $f$ admits infinitely many possible C-pair targets, and thus no meaningful finiteness statement can hold.

One could attempt to remedy this by restricting to those C-pairs satisfying $\Delta_Y = \Delta_f$, but then we might as well just work with the orbifold base itself.
Restricting to surjective maps is another option, but this is an unnatural limitation (cf. Remark \ref{remark:severi_curves}) and still does not prevent a single rational map from giving rise to several non-isomorphic C-pair targets.
\end{remark}

\bibliography{refs_severi}{}
\bibliographystyle{alpha}

\end{document}